\documentclass{amsart}
\usepackage{amssymb}
\usepackage{amsmath}
\usepackage{amscd}
\usepackage{url}
\usepackage{bbm}

\newtheorem{prop}{Proposition}[section]

\newtheorem{conj}[prop]{Conjecture}

\newtheorem{lem}[prop]{Lemma}
\newtheorem{thm}[prop]{Theorem}
\newtheorem{cor}[prop]{Corollary}
\theoremstyle{definition}

\newtheorem{remar}[prop]{Remark}

\newcommand{\Aut}{{\mathrm {Aut}}}

\renewcommand{\Im}{{\mathrm {Im}}}
\newcommand{\ord}{{\mathrm {ord}}}

\newcommand{\Nm}{{\mathrm {N}}}

\newcommand{\tr}{{\mathrm {tr}}}

\newcommand{\Sym}{{\mathrm {Sym}}}

\newcommand{\st}{{\mathrm {st}}}

\newcommand{\Spin}{{\mathrm {Spin}}}

\newcommand{\Frob}{{\mathrm {Frob}}}

\newcommand{\Gal}{\mathrm {Gal}}

\newcommand{\ad}{{\mathrm {ad}}}

\newcommand{\diag}{\mathrm{diag}}

\newcommand{\A}{{\mathbb A}}

\newcommand{\C}{{\mathbb C}}

\newcommand{\R}{{\mathbb R}}

\newcommand{\Q}{{\mathbb Q}}
\newcommand{\ZZ}{{\mathbb Z}}
\newcommand{\Z}{{\mathbb Z}}

\newcommand{\HH}{{\mathfrak H}}
\newcommand{\FFF}{{\mathcal F}}

\newcommand{\OO}{{\mathcal O}}

\newcommand{\pp}{{\mathfrak p}}

\newcommand{\qq}{{\mathfrak q}}

\newcommand{\p}{{\mathfrak p}}
\newcommand{\x}{{\mathbf{x}}}
\newcommand{\y}{{\mathbf{y}}}

\newcommand{\FF}{{\mathbb F}}

\newcommand{\GL}{\mathrm {GL}}

\newcommand{\SL}{\mathrm {SL}}
\newcommand{\Sp}{\mathrm {Sp}}

\newcommand{\SO}{\mathrm {SO}}
\newcommand{\GSp}{\mathrm {GSp}}

\newcommand{\Qbar}{\overline{\mathbb Q}}

\newcommand{\rhobar}{\overline{\rho}}

\begin{document}
\title{Automorphic forms for some even unimodular lattices}
\author{Neil Dummigan}
\author{Dan Fretwell}

\date{January 5th, 2021.}
\address{University of Sheffield\\ School of Mathematics and Statistics\\
Hicks Building\\ Hounsfield Road\\ Sheffield, S3 7RH\\
U.K.}
\email{n.p.dummigan@shef.ac.uk}
\address{School of Mathematics\\University Walk\\Bristol\\BS8 1TW\\U.K.}
\email{daniel.fretwell@bristol.ac.uk}
\begin{abstract}
We look at genera of even unimodular lattices of rank $12$ over the ring of integers of $\Q(\sqrt{5})$ and of rank $8$ over the ring of integers of $\Q(\sqrt{3})$, using Kneser neighbours to diagonalise spaces of scalar-valued algebraic modular forms. We conjecture most of the global Arthur parameters, and prove several of them using theta series, in the manner of Ikeda and Yamana. We find instances of  congruences for non-parallel weight Hilbert modular forms. Turning to the genus of Hermitian lattices of rank $12$ over the Eisenstein integers, even and unimodular over $\Z$, we prove a conjecture of Hentschel, Krieg and Nebe, identifying a certain linear combination of theta series as an Hermitian Ikeda lift, and we prove that another is an Hermitian Miyawaki lift.
\end{abstract}

\subjclass[2010]{11F41,11F27, 11F33, 11E12, 11E39}

\keywords{Algebraic modular forms, even unimodular lattices, theta series, Hilbert modular forms, Hermitian modular forms}

\maketitle

\section{Introduction}
Nebe and Venkov \cite{NV} looked at formal linear combinations of the $24$ Niemeier lattices, which represent classes in the genus of even, unimodular, Euclidean lattices of rank $24$. They found a set of $24$ eigenvectors for the action of an adjacency operator for Kneser $2$-neighbours, with distinct integer eigenvalues. This is equivalent to computing a set of Hecke eigenforms in a space of scalar-valued modular forms for a definite orthogonal group $\mathrm{O}_{24}$. They conjectured the degrees $g_i$ in which the Siegel theta series $\Theta^{(g_i)}(v_i)$ of these eigenvectors are first non-vanishing, and proved them in $22$ out of the $24$ cases.

Ikeda \cite[\S 7]{I2} identified $\Theta^{(g_i)}(v_i)$ in terms of Ikeda lifts and Miyawaki lifts, in $20$ out of the $24$ cases, exploiting his integral construction of Miyawaki lifts. Chenevier and Lannes \cite{CL} expanded upon his work and showed how it can be used to determine the global Arthur parameters of the automorphic representations $\pi_i$ of $\mathrm{O}_{24}(\A)$
generated by the $v_i$ in those $20$ cases. They also used different methods, based on Arthur's multiplicity formula, to recover the global Arthur parameters of all $24$ of the $\pi_i$, and completed the proof of Nebe and Venkov's conjecture on the degrees.

Ikeda and Yamana \cite{IY} constructed Ikeda lifts in the case of Hilbert modular forms over totally real fields. An integral construction of Miyawaki lifts based on this has been worked out in detail by Atobe \cite{At}. As an application, Ikeda and Yamana considered the genus of $6$ classes of even, unimodular lattices of rank $8$ over the ring of integers of the real quadratic field $E=\Q(\sqrt{2})$. They found a set of $6$ eigenvectors for the action of an adjacency operator for Kneser $\sqrt{2}$-neighbours, and determined the first non-vanishing theta series for each one, again using Ikeda and Miyawaki lifts, and for the latter a kind of triple product of eigenvectors introduced by Nebe and Venkov. The global Arthur parameters of the associated automorphic representations of $\mathrm{O}_{8}(\A_E)$ may be deduced from their results.

We extend this work of Ikeda and Yamana to other cases, in particular to the genus of $15$ classes of even, unimodular lattices of rank $12$ over the ring of integers of $E=\Q(\sqrt{5})$, first studied by Costello and Hsia \cite{CH}. We are able to conjecture the global Arthur parameters for $12$ out of the $15$ associated automorphic representations. These are formal direct sums of certain discrete automorphic representations of $\GL_m(\A_E)$, for various $m$. The ingredients going into these include representations of $\GL_2(\A_E)$ attached to Hilbert modular forms for $\SL_2(\OO_E)$, including examples of non-parallel weights, and symmetric square lifts to $\GL_3(\A_E)$. The conjectured global Arthur parameters are such that the implied eigenvalues for the Hecke operators $T_{(\sqrt{5})}$ and $T_{(2)}$ match those we computed using Kneser neighbours. They also satisfy the requirements of the Langlands parameters at the infinite places.

In $10$ of these $12$ cases we prove the conjecture for the global Arthur parameters, in Proposition \ref{rt5N12}. In one case we can apply directly a theorem of Ikeda and Yamana (Proposition \ref{ilifttheta}) to identify the global Arthur parameter and (upon checking the non-vanishing of a certain $L$-value) to determine the first non-vanishing theta series as a specific Ikeda lift. In other cases we follow Ikeda and Yamana, in using Kuang's analogue \cite{Ku} of a well-known theorem of B\"ocherer, to establish that certain Hilbert-Siegel modular forms, including Ikeda lifts,  are in the images of theta maps. Our Hecke eigenvalue computations then determine which eigenvectors they come from. Following Chenevier and Lannes, we use a theorem of Rallis to deduce the global Arthur parameters from the theta series. Finally, in one case we use non-vanishing of a triple product of eigenvectors to show that the theta series of a certain eigenvector is not orthogonal to a certain Miyawaki lift, which is enough to determine the global Arthur parameter, and we show that in fact the theta series {\em is} the Miyawaki lift.

An interesting aspect of the work of Chenevier and Lannes was the study of easily-proved congruences of Hecke eigenvalues between computed eigenvectors. Some could be accounted for, via the global Arthur parameters, by well-known congruences between genus-$1$ cusp forms and Eisenstein series, such as Ramanujan's mod $691$ congruence. Another was used to prove a mod $41$ congruence of Hecke eigenvalues involving genus-$1$ and vector-valued genus-$2$ forms, the first known instance of Harder's conjecture. In our case of rank $12$ for $\Q(\sqrt{5})$, we likewise observe congruences that can be explained in terms of congruences between Hilbert modular cusp forms and Eisenstein series, modulo prime divisors occurring in Dedekind zeta values. We also see two apparent congruences involving genus-$2$ vector-valued forms ``lifted'' from Hilbert modular forms (for us of non-parallel weight) in the manner of Johnson-Leung and Roberts \cite{JR}. The congruences are akin to those between cusp forms and Klingen-Eisenstein series. The moduli are ``dihedral'' congruence primes for certain cusp forms with quadratic character for $\Gamma_0(5)$. This leads us to a conjecture (\ref{hilbcong}) about congruences for non-parallel weight Hilbert modular forms. H. Hida has informed us that experimental instances of such congruences were discovered by H. Naganuma more than thirty years ago. We are not aware of them having been published anywhere before now.

We consider also the genus of $31$ classes of even, unimodular lattices of rank $8$ over the ring of integers of $E=\Q(\sqrt{3})$, first studied by Hung \cite{Hu}. We are able to conjecture the global Arthur parameters in $28$ out of the $31$ cases, and can prove $16$ of these. A new feature here is that the narrow class number of $\Q(\sqrt{3})$ is $2$ (whereas for both $\Q(\sqrt{5})$ and $\Q(\sqrt{2})$ it is $1$). Thus the quadratic character, and CM forms, associated to the narrow Hilbert class field $H=\Q(\zeta_{12})$, make an appearance. (Since $H/E$ is ramified only at infinite places, an unramified Hecke character for $H$ produces a level $1$ Hilbert modular form for $E$.) For $E=\Q(\sqrt{3})$, as for any $E=\Q(\sqrt{D})$ with squarefree $D=-1+4t$, the rank only has to be divisible by $2$ (indeed $\begin{pmatrix} 2 & \sqrt{D}\\\sqrt{D} & 2t\end{pmatrix}$ is even, unimodular of rank $2$, as pointed out in \cite{Hs}) and we look also at the baby cases of ranks $2,4$ and $6$.

Hentschel, Krieg and Nebe \cite{HKN} studied a genus of $5$ classes of Hermitian lattices of rank $12$ over the ring of integers of $E=\Q(\sqrt{-3})$, even and unimodular over $\Z$. The Hecke operator $T_{(2)}$ on the associated space of algebraic modular forms was diagonalised in \cite{DS}. In Proposition \ref{propherm}, for each eigenspace we determine the first non-vanishing (Hermitian) theta series, in particular confirming a conjecture of Hentschel, Krieg and Nebe that one of them is a degree-$4$ Hermitian Ikeda lift (up to scaling). We also identify one as an Hermitian Miyawaki lift, as studied by Atobe and Kojima \cite{AK}. For our purposes, we put together an Hermitian analogue of B\"ocherer's theorem (Proposition \ref{eigenthetaHerm}(3)), making use of some work of Lanphier and Urtis \cite{LU}, among others. To get from theta series to global Arthur parameters, the analogue of Rallis's theorem that we need (Proposition \ref{eigenthetaHerm}(1),(2)) is covered by work of Y. Liu \cite{Liu}.

In \S 2 we introduce some preliminaries on even unimodular lattices (over $\ZZ$), algebraic modular forms, local Langlands parameters, global Arthur parameters, theta series, Ikeda and Miyawaki lifts. In \S 3 we review briefly the work of Chenevier and Lannes on the Niemeier lattices. After some preliminaries in \S 4 on even unimodular lattices over real quadratic fields, in \S 5 we review the work of Ikeda and Yamana on $\Q(\sqrt{2})$. In \S 6 we further warm up with even unimodular lattices of rank $8$ for $\Q(\sqrt{5})$, where there are only $2$ classes in the genus. \S 7 deals with the more substantial case of the $15$ classes for rank $12$ for $\Q(\sqrt{5})$. We introduce the Hilbert modular forms involved, before presenting the Hecke eigenvalues for $T_{(\sqrt{5})}$ and $T_{(2)}$, conjecturing the global Arthur parameters, and proving what we can about them and the degrees via theta series. 
Then we look at the congruences mentioned above. \S 8 is about $E=\Q(\sqrt{3})$.
In \S 9 we consider to what extent we have covered all the interesting examples amenable to computation, and have a brief look at one or two more, with $E=\Q(\sqrt{7})$ and $\Q(\sqrt{11})$. After preliminaries in \S 10 on Hermitian lattices, even and unimodular over $\Z$, in \S 11 we look at the case $E=\Q(\sqrt{-3})$, rank $12$.

All the computed neighbour matrices used but not included in the paper, and their characteristic polynomials, may be found at the second-named author's webpage \url{https://www.danfretwell.com/kneser}.

We are grateful to G. Chenevier for his suggestion, in response to \cite{DS}, to adapt the methods of Ikeda \cite[\S 7]{I2} to Hermitian lattices. We thank him, O. Ta\"ibi and an anonymous referee for their comments on an earlier version of this paper. We thank also H. Hida for informing us of the work of Naganuma, M. Kirschmer, for advice on using his Magma code for neighbours over number fields, and for making some useful additions to it, and S. Yamana for his invaluable help with the proof of Proposition \ref{rt5N12}, case $\mathbf{i=13}$.

\section{Preliminaries}
\subsection{Even unimodular lattices and algebraic modular forms}
Let $L$ be a $\Z$-lattice in $V\simeq \Q^N$, with positive-definite integral quadratic form $\x\mapsto q_A(\x):=\frac{1}{2}\langle \x,\x\rangle$, where $A$ is a positive-definite symmetric matrix of size $N$ with rational entries and $\langle \y,\x\rangle:={}^t\y A\x$, for all $\x,\y\in V$. Associated to $L$ is an orthogonal group-scheme $O_L$, where for any commutative ring $R$,
$$O_L(R)=\{g\in \GL(L\otimes R)\mid q_A\circ g=q_A\}.$$ If $\A_f$ is the ring of finite adeles of $\Q$ then $O_L(\A_f)$ produces other lattices from $L$: given $(g_p)\in O_L(\A_f)$,
$(g_p)L:=V\cap((g_p)(L\otimes\A_f))$. These lattices are everywhere locally isometric to $L$, and form the {\em genus } of $L$. Let $K=\prod_p O_L(\Z_p)\subset O_L(\A_f)$ be the stabiliser of $L$. Then there is a natural bijection between $C_L:=O_L(\Q)\backslash O_L(\A_f)/K$ and the set of classes in the genus of $L$, which is finite, say represented by classes $[L_1],\ldots,[L_h]$, with $[L]=[L_1]$.

The set of $\C$-valued functions on $C_L$ may be regarded as the space of functions on $O_L(\A)$, left-invariant under $O_L(\Q)$, right-invariant under $K$ and transforming on the right via the trivial representation of $O_L(\R)$. Thus they are scalar-valued algebraic modular forms for $O_L$, forming a space denoted $M(\C,K)$. It is acted upon by the Hecke algebra $H_{K}$ of all locally constant, compactly supported functions $O_L(\A_f)\rightarrow\C$ that are left and right $K$-invariant. It is a semi-simple module for $H_K$ \cite[Prop. 6.11]{Gr1}, and there is a natural bijection between simple $H_K$-submodules of $M(\C,K)$ and irreducible automorphic representations of $O_L(\A)$ with a $K$-fixed vector and such that $\pi_{\infty}$ is trivial \cite[Proposition 2.5]{GV}.

We now suppose that $L$ is even integral ($\langle \x,\x\rangle\in 2\ZZ\,\,\,\forall \x\in L$) and unimodular ($L^*=L$, where $L^*:=\{\y\in V|\,\,\langle \y,\x\rangle\in\ZZ\,\,\forall\x\in L\}$). (By adjusting $A$, we may suppose that $L=\Z^N$, then $A$ has integer entries, even on the diagonal, and determinant $1$.) Then $8\mid N$ \cite[Scholium 2.2.2(b)]{CL} and every even unimodular lattice of rank $N$ is equivalent to one in the genus of $L$ \cite[Chapter 15, \S 7]{CS}. At all primes $p$, $A$ is equivalent over $\Z_p$ to $\begin{pmatrix}0_{N/2} & I_{N/2}\\I_{N/2} & 0_{N/2}\end{pmatrix}$ \cite[Scholium 2.2.5]{CL}. Hence $\SO_L/\Z_p$ is reductive and $\SO_L(\Q_p)$ is a split orthogonal group, with $\SO_L(\Z_p)$ a hyperspecial maximal compact subgroup. To deal with $p=2$, we have to define the group scheme $\SO_L/\Z$ as the kernel of the Dickson determinant on $O_L$.
As explained just before \cite[4.2.11]{CL}, the $p$-component of $H_K$ is a subring of a Hecke algebra for $\SO_L(\Q_p)$ with respect to $\SO_L(\Z_p)$.  Convolution by the indicator function of the double coset $K\diag(p,1,\ldots,1,p^{-1},1,\ldots,1)K$ gives a Hecke operator denoted $T_p$, which can be made explicit using the notion of Kneser $p$-neighbours \cite[6.2.8]{CL}. Given lattices $M$ and $M'$ in $V$, we say that $M'$ is a $p$-neighbour of $M$ if $\#\left(\frac{M}{M\cap M'}\right) =\#\left(\frac{M'}{M\cap M'}\right)=p$. The number of $p$-neighbours of $M$ is finite, equal to the number of left cosets of $K$ into which $K\diag(p,1,\ldots,1,p^{-1},1,\ldots,1)K$ decomposes, and if $M'$ is a $p$-neighbour of $M$ then $M$ and $M'$ belong to the same genus. The Hecke operator $T_p$ is represented, with respect to the basis $\{e_1,\ldots,e_h\}$ of $M(\C,K)$, where $e_i([L_j])=\delta_{ij}$, by the matrix $(b_{ij})$, where among the Kneser $p$-neighbours of $L_i$, $b_{ij}$ is the number isometric to $L_j$. The Hecke algebra is commutative \cite[Proposition 2.10]{Gr2}, and there exists a basis of $M(\C,K)$ of simultaneous eigenvectors for $H_K$. Let $v_i$ and $\pi_i$ be the corresponding eigenvectors and automorphic representations, respectively, in some order for $1\leq i\leq h$.

\subsection{Local Langlands parameters}
For each local Weil group $W_{\R}$ and $W_{\Q_p}$ of $\Q$ there is associated to $\pi_i$ a Langlands parameter, a homomorphism $c_{\infty}(\pi_i)$ or $c_p(\pi_i)$ from that group to the Langlands dual group $O_{N}(\C)$ of $O_L$. (As explained in \cite[6.4.7]{CL}, it lands in $\SO_N(\C)$ but is only defined up to conjugation by $O_N(\C)$.)
 Now $W_{\C}=\C^{\times}$ is a subgroup of index $2$ in $W_{\R}$, and it is a consequence of the fact that $v_i$ is scalar-valued that (up to conjugation)
 $$c_{\infty}(\pi_i):z\mapsto $$ $$\diag\left((z/\overline{z})^{(N/2)-1},(z/\overline{z})^{(N/2)-2},\ldots,(z/\overline{z})^{0},(z/\overline{z})^{1-(N/2)},(z/\overline{z})^{2-(N/2)},\ldots,(z/\overline{z})^{0}\right).$$
 At any finite prime $p$, since in our situation $\pi_i$ is unramified at $p$, $c_p(\pi_i)$ is determined by $\mathrm{Frob}_p\mapsto t_p(\pi_i)$, the Satake parameter at $p$, in fact this is how we know it exists without assuming the local Langlands conjecture for $O_N(\Q_p)$.
 This determines $\lambda_i(T_p)$, by the formula (cf. \cite[(3.13)]{Gr2})
 \begin{equation}\label{Gross} \lambda_i(T_p)=p^{(N/2)-1}\tr(t_p(\pi_i)).\end{equation}

\subsection{Global Arthur parameters}\label{GAP}
A complete description of those automorphic representations, of a split special orthogonal group $G^*$, occurring discretely in $L^2(G^*(\Q)\backslash G^*(\A))$, was given by Arthur \cite{Ar}. This was extended to a wider class of special orthogonal groups
 (including $\SO_L$) by Ta\"ibi \cite{Tai}. (The representations of $O_L(\A)$ we are looking at are classified in terms of their restriction to $\SO_L(\A)$, as explained in \cite[6.4.7]{CL}, and they also satisfy the regularity condition in the work of Arthur and Ta\"ibi.) Part of this description is that to such an automorphic representation is attached a ``global Arthur parameter'', a formal unordered sum of the form $\oplus_{k=1}^m\Pi_k[d_k]$, where $\Pi_k$ is a cuspidal automorphic representation of $\GL_{n_k}(\A)$, $d_k\geq 1$ and $\sum_{k=1}^mn_kd_k=N$. For each $\Pi_k$ there are local Langlands parameters $c_{\infty}:W_{\R}\rightarrow\GL_{n_k}(\C)$ and $c_p:W_{\Q_p}\rightarrow\GL_{n_k}(\C)$ ($\Frob_p\mapsto t_p(\Pi_k)$), defined up to conjugation in the codomain. For us there are four cases:
 \begin{enumerate}
 \item $n_k=1$ and $\Pi_k$ is trivial;
 \item $n_k=2$, $c_{\infty}(\Pi_k)(z)=\diag((z/\overline{z})^{a/2},(z/\overline{z})^{-a/2})$, and $\Pi_k$, denoted $\Delta_a$, is the automorphic representation generated by a cusp form $f$ of weight $\kappa$, with $a=\kappa-1$. If $a_p(f)$ is the Hecke eigenvalue at $p$ then $t_p(\Pi_k)=\diag(\alpha,\alpha^{-1})$, with $a_p(f)=p^{(\kappa-1)/2}(\alpha+\alpha^{-1})$;
 \item $n_k=3$, $c_{\infty}(\Pi_k)(z)=\diag((z/\overline {z})^{a},1,(z/\overline{z})^{-a})$, and $\Pi_k$, denoted $\Sym^2\Delta_a$, is the symmetric square lift of $\Delta_a$;
 \item $n_k=4$, $c_{\infty}(\Pi_k)(z)=\diag((z/\overline{z})^{a/2},(z/\overline{z})^{b/2},(z/\overline{z})^{-b/2},(z/\overline{z})^{-a/2})$, and $\Pi_k$, denoted $\Delta_{a,b}$, is the spinor lift to $\GL_4(\A)$ of the automorphic representation of $\GSp_2(\A)$ generated by a Siegel cusp form $F$ of weight $(j,\kappa)$ (vector-valued when $j>0$), with $a=j+2\kappa-3, b=j+1$. Note that $j$ is even, so $a,b$ are odd.
      \end{enumerate}
     Letting $Z$ denote the centre of $\GL_{n_kd_k}$, the representation $\Pi_k[d_k]$ of $\GL_{n_kd_k}(\A)$ occurs discretely in $L^2(Z(\A)\GL_{n_kd_k}(\Q)\backslash\GL_{n_kd_k}(\A))$.
 In all cases, $c_{\infty}(\Pi_k[d_k])(z)$ $$=c_{\infty}(\Pi_k)(z)\otimes\diag((z/\overline z)^{(d_k-1)/2},(z/\overline z)^{(d_k-3)/2},\ldots,(z/\overline z)^{(3-d_k)/2},(z/\overline z)^{(1-d_k)/2})$$ and
 $$t_p(\Pi_k[d_k])=t_p(\Pi_k)\otimes\diag(p^{(d_k-1)/2},p^{(d_k-3)/2},\ldots,p^{(3-d_k)/2},p^{(1-d_k)/2}).$$
 When $\Pi_k$ is the trivial representation of $\GL_1(\A)$, the representation $\Pi_k[d_k]$ of $\GL_{d_k}(\A)$ is written simply $[d_k]$. When direct summing the $\Pi_k[d_k]$, we direct sum the associated local Langlands parameters. To say that $\oplus_{k=1}^m\Pi_k[d_k]$ is the global Arthur parameter of $\pi_i$ is to say that each $c_p(\pi_i)$ and $c_{\infty}(\pi_i)$, composed with the standard representation from $\SO_N(\C)$ to $\GL_N(\C)$, is conjugate in $\GL_N(\C)$ to the local Langlands parameter associated to $\oplus_{k=1}^m\Pi_k[d_k]$.

 \subsection{Theta series}
 Let $L$ be an even unimodular lattice in $\Q^N$, as above, and for each $m\geq 1$ define its theta series of degree $m$ by
 $$\theta^{(m)}(L,Z):=\sum_{\x\in L^m}\exp(\pi i\tr(\langle\x,\x\rangle Z)),$$ where $Z\in\HH_m:=\{Z\in M_m(\C):\,\,{}^tZ=Z,\,\Im(Z)>0\}$, the Siegel upper half space of degree $m$. It is known that $\theta^{(m)}(L)$ is a Siegel modular form of weight $N/2$ for the full modular group $\Sp_m(\Z):=\{g\in M_{2m}(\Z):\,\,{}^tgJg=J\}$, where $J=\begin{pmatrix} 0_m & -I_m\\I_m & 0_m\end{pmatrix}$. If $g=\begin{pmatrix} A & B\\C & D\end{pmatrix}\in\Sp_m(\Z)$ then $$\theta^{(m)}(L,(AZ+B)(CZ+D)^{-1})=\det(CZ+D)^{N/2}\theta^{(m)}(L,Z).$$
 Now one can define linear maps $\Theta^{(m)}: M(\C,K)\rightarrow M_{N/2}(\Sp_m(\Z))$ by
 $$\Theta^{(m)}\left(\sum_{j=1}^h x_je_j\right):=\sum_{j=1}^h\frac{x_j}{|\mathrm{Aut}(L_j)|}\,\theta^{(m)}(L_j),$$
 where $e_i([L_j])=\delta_{ij}$.
 \begin{prop}\label{eigentheta}
 \begin{enumerate}
 \item If $v_i\in M(\C,K)$ is an eigenvector for $H_K$, then $\Theta^{(m)}(v_i)$ (if non-zero) is a Hecke eigenform.
 \item Suppose that $\Theta^{(m)}(v_i)$ is non-zero, and that $(N/2)\geq m$. Let $t_p(\pi)=\diag(\beta_{1,p},\ldots,\beta_{N/2,p},\beta_{1,p}^{-1},\ldots,\beta_{N/2,p}^{-1})$ be the Satake parameter at $p$ for $v_i$, and let $(\diag(\alpha_{1,p},\ldots,\alpha_{m,p},1,\alpha_{1,p}^{-1},\ldots,\alpha_{m,p}^{-1})\in\SO(m+1,m)(\C)$ be the Satake parameter at $p$ of the automorphic representation of $\Sp_m(\A)$ generated by $\Theta^{(m)}(v_i)$. Then, as multisets,
     $$\{\beta_{1,p}^{\pm 1},\ldots,\beta_{(N/2),p}^{\pm 1}\}=$$
     $$\begin{cases} \{\alpha_{1,p}^{\pm 1},\ldots,\alpha_{m,p}^{\pm 1}\}\cup \{p^{\pm((N/2)-m-1)},\ldots,p^{\pm 1},1,1\} & \text{if $(N/2)>m$};\\\{\alpha_{1,p}^{\pm 1},\ldots,\alpha_{m,p}^{\pm 1}\} & \text{ if $(N/2)=m$. }\end{cases}$$
     \item If $8\mid N$ and $(N/2)\geq m+1$, a cuspidal Hecke eigenform $F\in S_{N/2}(\Sp_m(\Z))$ is in the image of $\Theta^{(m)}$ if and only if $L(\st,F,(N/2)-m)\neq 0$, where $L(\st,F,s)=\prod_p\left((1-p^{-s})^{-1}\prod_{i=1}^m((1-\alpha_{i,p}p^{-s})(1-\alpha_{i,p}^{-1}p^{-s}))^{-1}\right)$ is the standard $L$-function.
 \end{enumerate}
 \end{prop}
 (1) and (2) follow from a theorem of Rallis \cite[Remark 4.4(A)]{R}, as explained in \cite[7.1]{CL}. (3) is a theorem of B\"ocherer \cite[Theorem $4_1$]{Bo}.

 The {\em degree} of $v_i$ is defined to be the smallest $m$ such that $\Theta^{(m)}(v_i)\neq 0$. Note that if $m\geq 1$, $\Phi(\Theta^{m}(v_i))=\Theta^{(m-1)}(v_i)$, where $\Phi$ is the Siegel operator, so this first non-zero theta series is cuspidal, except in the case that $v_i$ is a multiple of the all-ones vector, where $\Theta^{(m)}(v_i)$ is an Eisenstein series for all $1\leq m<(N/2)$, by Siegel's Main Theorem, and by convention the degree of $v_i$ is $0$.

 Following Nebe and Venkov, but with slightly different normalisation as in \cite[\S 12.5]{IY}, we define an inner product and multiplication on $M(\C,K)$ by
 $$(e_i,e_j):=\frac{1}{|\mathrm{Aut}(L_i)|}\delta_{ij}$$ and
 $$e_i\circ e_j:=\delta_{ij}e_i.$$
 Let $g_i$ be the degree of $v_i$, and let $F_i:=\Theta^{(g_i)}(v_i)$.
 The following is equivalent to \cite[Lemma 7.1]{I2}.
 \begin{prop}\label{nonzero}
 $$\langle \Theta^{(g_i+g_j)}(v_k)|_{\HH_{g_i}\times\HH_{g_j}},F_i\times F_j\rangle =\frac{\langle F_i,F_i\rangle\langle F_j,F_j\rangle}{(v_i,v_i)(v_j,v_j)}\,(v_k,v_i\circ v_j).$$
 In particular, $(v_k,v_i\circ v_j)\neq 0$ if and only if the left hand side is non-zero.
 \end{prop}
 \begin{cor}\label{degbound}
 If $(v_k,v_i\circ v_j)\neq 0$ then $g_k\leq g_i+g_j$.
 \end{cor}
 See \cite[Proposition 2.3]{NV} for an alternative approach. If $v_i=\sum_{t=1}^hc_{it}e_t$ then $(v_k,v_i\circ v_j)=\sum_{t=1}^h \frac{1}{|\mathrm{Aut}(L_t)|}\,c_{kt}c_{it}c_{jt},$ so it is easy to compute in any given case.

 \subsection{Ikeda and Miyawaki lifts}
 \begin{prop}\label{IkMiy} Let $\kappa, g$ be even natural numbers. Let $f\in S_{2\kappa-g}(\SL_2(\Z))$ be a normalised Hecke eigenform. Let $G\in S_{\kappa}(\Sp_r(\Z))$ be a Hecke eigenform, for $r<g$.
 \begin{enumerate}
 \item There exists a Hecke eigenform $F\in S_{\kappa}(\Sp_g(\Z))$ with standard $L$-function $$L(\st,F,s)=\zeta(s)\prod_{i=1}^g L(f,s+\kappa-i).$$
 \item The function $$\FFF_{f,G}(Z):=\int_{\Sp_r(\Z)\backslash\HH_r}F\left(\begin{pmatrix} Z & 0\\0 & W\end{pmatrix}\right)G(-\overline{W})(\det\Im W)^{\kappa-r-1}\,dW,$$ if non-zero, is a Hecke eigenform in $S_{\kappa}(\Sp_{g-r})$, with standard $L$-function $$L(\st,\FFF_{f,G},s)=L(\st,G,s)\prod_{i=1}^{g-2r}L(f,s+\kappa-r-i).$$
 \end{enumerate}
 \end{prop}
 (1) is a theorem of Ikeda \cite{I3}, and $F$ (whose existence was conjectured by Duke and Imamoglu) is the Ikeda lift $I^{(g)}(f)$. Its scaling is determined naturally by a choice of scaling of a half-integral weight form in Kohnen's plus space corresponding to $f$. (2) was also proved by Ikeda \cite{I2}, and gives his construction of a form whose existence was conjectured by Miyawaki in the case $g=4, r=1$ \cite{Mi}.
\section{Even unimodular $24$-dimensional quadratic forms over $\Q$}
In the case $N=24$, the genus of even unimodular lattices has $h=24$ classes, represented by the Niemeier lattices. Nebe and Venkov diagonalised the operator $T_2$, and found that it has $24$ distinct rational integer eigenvalues, shown in the table below \cite{NV}. We have listed the eigenvalues $\lambda_i(T_2)$ in descending order, for $1\leq i\leq 24$. Let $v_i$ and $\pi_i$ be the corresponding eigenvectors and automorphic representations, respectively. Chenevier and Lannes determined the $\pi_i$ in terms of Arthur's endoscopic classification of automorphic representations of classical groups \cite{CL}. The global Arthur parameters are listed in the final column of the table. Each one $A_i=\oplus_{k=1}^m\Pi_k[d_k]$ must pass the two tests that
$$c_{\infty}(A_i)(z)=\diag\left((z/\overline{z})^{11},(z/\overline{z})^{10},\ldots,(z/\overline{z})^{0},(z/\overline{z})^{-11},(z/\overline{z})^{-10},\ldots,(z/\overline{z})^{0}\right)$$
    and that $2^{11}\tr(t_2(A_i))=\lambda_i(T_2)$, as computed using neighbours. That would be enough to justify a conjecture that these global Arthur parameters are correct, but Chenevier and Lannes gave several proofs that they really are correct, for example by using Arthur's multiplicity formula applied to the group $\SO_{24}$.
\vskip10pt
\begin{center}
\begin{tabular}{|c|c|c|c|}\hline$\mathbf{i}$ & $\lambda_i\left(T_{2}\right)$ & degree & Global Arthur parameters \\\hline
 $\mathbf{1}$ & $8390655$ & $0$ & $[23]\oplus[1]$\\$\mathbf{2}$ & $4192830$ & $1$ & $\Sym^2\Delta_{11}\oplus[21]$\\$\mathbf{3}$ & $2098332$ & $2$ & $\Delta_{21}[2]\oplus[1]\oplus[19]$\\$\mathbf{4}$ & $1049832$ & $3$ & $\Sym^2\Delta_{11}\oplus\Delta_{19}[2]\oplus[17]$\\$\mathbf{5}$ & $533160$ & $4$ & $\Delta_{19}[4]\oplus[1]\oplus[15]$\\$\mathbf{6}$ & $519120$ & $4$ & $\Delta_{21}[2]\oplus\Delta_{17}[2]\oplus[1]\oplus[15]$\\$\mathbf{7}$ & $268560$ & $5$ & $\Sym^2\Delta_{11}\oplus\Delta_{19}[2]\oplus\Delta_{15}[2]\oplus[13]$\\$\mathbf{8}$ & $244800$ & $5$ & $\Sym^2\Delta_{11}\oplus\Delta_{17}[4]\oplus[13]$\\$\mathbf{9}$ & $145152$ & $6$ & $\Delta_{21}[2]\oplus\Delta_{15}[4]\oplus[1]\oplus[11]$\\$\mathbf{10}$ & $126000$ & $6$ & $\Delta_{21,13}[2]\oplus\Delta_{17}[2]\oplus[1]\oplus[11]$\\$\mathbf{11}$ & $99792$ & $6$ & $\Delta_{17}[6]\oplus[1]\oplus[11]$\\$\mathbf{12}$ & $91152$ & $7$ & $\Sym^2\Delta_{11}\oplus\Delta_{15}[6]\oplus[9]$\\$\mathbf{13}$ & $89640$ & $8$ & $\Delta_{15}[8]\oplus[1]\oplus[7]$\\$\mathbf{14}$ & $69552$ & $7$ & $\Sym^2\Delta_{11}\oplus\Delta_{19}[2]\oplus\Delta_{15}[2]\oplus\Delta_{11}[2]\oplus[9]$\\$\mathbf{15}$ & $51552$ & $8$ & $\Delta_{21,9}[2]\oplus\Delta_{15}[4]\oplus[1]\oplus[7]$\\$\mathbf{16}$ & $45792$ & $7$ & $\Sym^2\Delta_{11}\oplus\Delta_{17}[4]\oplus\Delta_{11}[2]\oplus[9]$\\$\mathbf{17}$ & $35640$ & $8$ & $\Delta_{19}[4]\oplus\Delta_{11}[4]\oplus[1]\oplus[7]$\\$\mathbf{18}$ & $21600$ & $8$ & $\Delta_{21}[2]\oplus\Delta_{17}[2]\oplus\Delta_{11}[4]\oplus[1]\oplus[7]$\\$\mathbf{19}$ & $17280$ & $9$ & $\Sym^2\Delta_{11}\oplus\Delta_{19,7}[2]\oplus\Delta_{15}[2]\oplus\Delta_{11}[2]\oplus[5]$\\$\mathbf{20}$ & $5040$ & $9$ & $\Sym^2\Delta_{11}\oplus\Delta_{19}[2]\oplus\Delta_{11}[6]\oplus[5]$\\$\mathbf{21}$ & $-7920$ & $10$ & $\Delta_{21,5}[2]\oplus\Delta_{17}[2]\oplus\Delta_{11}[4]\oplus[1]\oplus[3]$\\$\mathbf{22}$ & $-16128$ & $10$ & $\Delta_{21}[2]\oplus\Delta_{11}[8]\oplus[1]\oplus[3]$\\$\mathbf{23}$ & $-48528$ & $11$ & $\Sym^2\Delta_{11}\oplus\Delta_{11}[10]\oplus[1]$\\$\mathbf{24}$ & $-98280$ & $12$ & $\Delta_{11}[12]$\\\hline
 \end{tabular}
 \end{center}
\vskip10pt
 The degrees were proved by Nebe and Venkov \cite{NV}, with the exception of cases $\mathbf{19}$ and $\mathbf{21}$, where the degrees they conjectured were later proved by Chenevier and Lannes \cite{CL}. As pointed out in \cite[1.4]{CL}, $20$ out of the $24$ global Arthur parameters (all those not involving any $\Delta_{a,b}$) may be proved as a direct consequence of work of Ikeda \cite[\S 7]{I2}. For these cases, he identified $\Theta^{(g_i)}(v_i)$ in terms of Ikeda lifts and Miyawaki lifts. For example, for $\mathbf{5}$, letting $\kappa=12$ and $g=4$, Proposition \ref{IkMiy}(1) gives us an Ikeda lift $F=I^{(4)}(f)\in S_{12}(\Sp_4(\Z))$, where $f\in S_{20}(\SL_2(\Z))$. Proposition \ref{eigentheta}(3) (B\"ocherer's Theorem), combined with $L(\st,F,s)=\zeta(s)\prod_{i=1}^g L(f,s+\kappa-i)$, shows that $F=\Theta^{(4)}(v)$ for some $v\in M(\C,K)$, necessarily an eigenvector $v=v_i$, using Proposition \ref{eigentheta}(1) and the fact that all the eigenspaces in $M(\C,K)$ are $1$-dimensional. The values of $\lambda_i(T_2)$ show that it can only be $i=5$. The formula $L(\st,F,s)=\zeta(s)\prod_{i=1}^g L(f,s+\kappa-i)$ implies Satake parameters for the associated automorphic representation of $\Sp_4(\A)$ that make $\Delta_{19}[4]\oplus[1]$ its global Arthur parameter. The extra $\oplus[15]$ in the global Arthur parameter of $\pi_5$ is accounted for by the extra $\{p^{\pm 7},\ldots,p^{\pm 1},1\}$ in Proposition \ref{eigentheta}(2) (Rallis's Theorem), with $N=24, m=4$.
\section{Preliminaries on even unimodular lattices over real quadratic fields}
Let $E$ be a real quadratic field, with ring of integers $\OO_E$. Let $L$ be an $\OO_E$-lattice in $V\simeq E^N$, with totally positive-definite quadratic form $\x\mapsto \frac{1}{2}\langle \x,\x\rangle$. We may define an orthogonal group scheme $O_L$ over $\OO_E$, a genus, algebraic modular forms $M(\C,K)$, Hecke operators $T_{\pp}$, $v_i$ and $\pi_i$ very much as before. We assume that $L$ is even ($\langle\x,\x\rangle\in 2\OO_E\,\,\forall\x\in L$), and unimodular ($L^*=L$, where $L^*:=\{\y\in V|\,\,\langle \y,\x\rangle\in\OO_E\,\,\forall\x\in L\}$). The following result of Scharlau is worth noting.
\begin{prop}\cite[Proposition 3.1]{Sc} If $4\mid N$, there is a unique genus of free, even unimodular lattices of determinant $1$.
\end{prop}
If $\OO_E$ has class number $1$, the word ``free'' is superfluous, and if the narrow class number is equal to the class number then ``determinant $1$'' is superfluous, since this is the determinant of a Gram matrix, well-defined modulo squares of units, but the determinant of a totally positive-definite unimodular lattice is a totally positive unit, necessarily a square under the given condition.

There will be local Langlands parameters $c_{\infty_1}(\pi_i),c_{\infty_2}(\pi_i):W_{\R}\rightarrow\SO_N(\C)$, for the two infinite places $\infty_1,\infty_2$, and $c_{\pp}(\pi_i):W_{E_{\pp}}\rightarrow\SO_N(\C)$ (with $\Frob_{\pp}\mapsto t_{\pp}(\pi_i)$) for each finite prime $\pp$. In the global Arthur parameters, cuspidal automorphic representations of $\GL_{n_k}(\A)$ are replaced by cuspidal automorphic representations of $\GL_{n_k}(\A_E)$, modular forms by Hilbert modular forms.
In order for everything to work as before, we must check in each case we look at that, for every finite prime $\pp$, $\SO_L/F_{\pp}$ is split and $\SO_L/\OO_{\pp}$ is reductive (hence, by \cite[3.8.1]{Ti}, $\SO_L(\OO_{\pp})$ is a hyperspecial maximal compact subgroup). This is necessary for the relation between $\pp$-neighbours and the Hecke operators $T_{\pp}$, for the equation (\ref{Gross}) for Hecke eigenvalues, and for the application of Rallis's theorem to Proposition \ref{eigentheta}.

If the norm of a fundamental unit is $-1$ (e.g if $\OO_E$ has narrow class number $1$, with every ideal generated by a totally positive element), then the different $\mathfrak{D}$ is generated by a totally positive element $\delta$. Let $\sigma_1,\sigma_2$ be the two real embedddings of $E$. We may define the theta series of degree $m$ of $L$ as
$$\theta^{(m)}(L)=\sum_{\x\in L^m}\exp\left(\pi i\tr\left(\sigma_1\left(\langle\x,\x\rangle/\delta\right) Z_1+\sigma_2\left(\langle\x,\x\rangle/\delta\right) Z_2\right)\right),$$ where $Z=(Z_1,Z_2)\in \HH_m^2$.

If the norm of a fundamental unit is $1$ then $\mathfrak{D}$ has a generator $\delta$ with $\sigma_1(\delta)>0$ and $\sigma_2(\delta)<0$, and we define $\theta^{(m)}(L)$ by the same formula, but now with $(Z_1,Z_2)\in\HH_m\times\HH_m^-$, where $\HH_m^-:=\{Z\in M_m(\C):\,\,{}^tZ=Z,\,\Im(Z)<0\}$. Then in either case $\theta^{(m)}(L)\in M_{N/2}(\Sp_m(\OO_E))$, where the $N/2$ is parallel weight $(N/2,N/2)$, cf. \cite[\S 4]{Hu},\cite[p.371]{HH}. Thus, if $g=\begin{pmatrix} A & B\\C & D\end{pmatrix}\in\Sp_m(\OO_E)$ and we denote $\sigma_1(A)=A_1$ etc., then $$\theta^{(m)}(L,(AZ+B)(CZ+D)^{-1})$$
$$=\det(C_1Z_1+D_1)^{N/2}\det(C_2Z_2+D_2)^{N/2}\theta^{(m)}(L,(Z_1,Z_2)),$$
where $$(AZ+B)(CZ+D)^{-1}:=((A_1Z_1+B_1)(C_1Z_1+D_1)^{-1},(A_2Z_2+B_2)(C_2Z_2+D_2)^{-1}).$$
 Again one can define linear maps $\Theta^{(m)}: M(\C,K)\rightarrow M_{N/2}(\Sp_m(\OO_E))$ by
 $$\Theta^{(m)}\left(\sum_{j=1}^h x_je_j\right):=\sum_{j=1}^h\frac{x_j}{|\mathrm{Aut}(L_j)|}\,\theta^{(m)}(L_j).$$
 Parts (1) and (2) of Proposition \ref{eigentheta} are just as before. Note that we are concerned with automorphic representations of $\Sp_m(\A_E)$, not $\GSp_m(\A_E)$ so we have strong approximation even when the narrow class number of $E$ is not $1$. Thus it makes sense to talk of an individual function $F$ on $\HH_m^2$ or $\HH_m\times \HH_m^-$ being a Hecke eigenform (interchangeable with an automorphic form on $\Sp_m(\A_E)$, as explained in \cite[p.926--7]{Ku}), but this does not include the Hecke operators usually denoted $T(\pp)$, which only exist for $\GSp_m$. In place of (3) we have
 \begin{prop}\label{eigentheta3}
 If $N/2>m+1$ (with $N$ such that we have an even unimodular lattice $L$, with reference to whose genus the maps $\Theta^{(m)}$ are defined) then a Hecke eigenform $F\in S_{N/2}(\Sp_m(\OO_E))$ is in the image of $\Theta^{(m)}$
 if $L(\st,F,(N/2)-m)\neq 0$.
 \end{prop}
 This is based on work of Kuang \cite{Ku}. We do not need his condition $8\mid N$, whose purpose was to construct something like an even unimodular quadratic form, given that we start with one. His Theorem 2 omits the condition $L(\st,F,(N/2)-m)\neq 0$, and his Proposition 5.4 appears to claim that the non-vanishing follows automatically from that of the local factors. But the example where $E=\Q$ (he works in the setting of any totally real field), $N=32$, $\kappa=16, m=g=14$, $f\in S_{18}(\SL_2(\Z))$ and $F=I^{(14)}(f)\in S_{16}(\Sp_{14}(\Z))$ shows that this is not so. Here $L(\st,F,2)=\zeta(2)\prod_{i=1}^{14} L(f,18-i),$ which includes the vanishing factor $L(f,9)$.

 The notion of degree, and Proposition \ref{nonzero}, carry over in the obvious fashion, as do the statements about Ikeda lifts and Miyawaki lifts. Ikeda lifts for Hilbert modular forms were constructed by Ikeda and Yamana \cite{IY}, and the application to Miyawaki lifts of Hilbert-Siegel modular forms has been worked out in detail by Atobe \cite{At}.
 The following is from Corollaries 1.4 and 11.3 in \cite{IY}.
 \begin{prop}\label{ilifttheta} Given $N, L$ and $M(\C,K)$ as above, if $E$ is of narrow class number $H=1$ suppose that $f\in S_{N/2}(\SL_2(\OO_E))$ is a Hecke eigenform, with associated cuspidal automorphic representation $\Delta_{(N/2)-1}$ of $\GL_2(\A_E)$. More generally, in place of $f$ consider the appropriate $H$-tuple of functions on $\HH^2$ representing an automorphic form on $\GL_2(\A_E)$ that is right-invariant under $\prod \GL_2(\OO_{\pp})$ and has components at the infinite places corresponding to weight $N/2$, say $f\in S_{N/2}(\GL_2(\A_E),\prod \GL_2(\OO_{\pp}))$.
 \begin{enumerate}
 \item There exists $\pi_i$ with global Arthur parameter $\Delta_{(N/2)-1}[N/2]$.
 \item If $L(f,N/4)\neq 0$ then $\Theta^{(N/2)}(v_i)=I^{(N/2)}(f)$, up to scalar multiples, whereas if $L(f,N/4)=0$ then $\Theta^{(N/2)}(v_i)=0$.
 \end{enumerate}
 \end{prop}
\section{Even unimodular $8$-dimensional quadratic forms over $\Q(\sqrt{2})$}
Takada \cite{Tak} showed that if $E=\Q(\sqrt{2})$ (for which $H=1$) then the genus of even unimodular $\OO_E$-lattices contains a single class if $N=4$ (in which case there will be a single $v_1=(1)$, $\pi_1$ of global Arthur parameter $[1]\oplus [3]$). Hsia and Hung \cite{HH} proved that there are $6$ classes if $N=8$.
These were considered by Ikeda and Yamana \cite[\S\S 12.4,12.5]{IY}. They took the matrix from \cite{HH} representing $T_{(\sqrt{2})}$ with respect to the basis $\{e_1,\ldots,e_h\}$ for $M(\C,K)$, and computed its eigenvalues and eigenvectors. The eigenvalues are in the table below. The global Arthur parameters follow, using Proposition \ref{eigentheta}(2), from their determination of all the $\Theta^{(g_i)}(v_i)$. (We know that $O_L$ is split over each $E_{\pp}$ and reductive over each $\OO_{\pp}$, since one choice of $L$ is $E_8\otimes_{\Z}\OO_E$.)

First, since $v_1={}^t(1,\ldots,1)$, $\theta^{(m)}(v_1)$ is an Eisenstein series for all $m$ with $1\leq m<(N/2)-1=3$, by the Siegel-Weil formula. The $[1]\oplus [7]$ then follows from Proposition \ref{eigentheta}(2). The space $S_4(\SL_2(\OO_E))$ is spanned by a single form $g$, with associated $\Delta_3$. Using Proposition \ref{ilifttheta}(1), there exists some $\pi_i$ with global Arthur parameter $\Delta_3[4]$, which can only be $\pi_2$, and if one wants the theta series too then Proposition \ref{ilifttheta}(2) gives $\Theta^{(4)}(v_2)=I^{(4)}(g)$. (Magma gives $L(g,2)\approx 0.440328\neq 0$.) The space $S_6(\SL_2(\OO_E))$ is spanned by Galois conjugate forms $f_1,f_2$, with associated cuspidal automorphic representations of $\GL_2(\A_E)$ both denoted $\Delta_5^{(2)}$. Both $I^{(2)}(f_1)$ and $I^{(2)}(f_2)$ are in the image of $\Theta^{(2)}$, by Proposition \ref{eigentheta3}. This accounts for $\pi_5$ and $\pi_6$. Similarly $g$ is in the image of $\Theta^{(1)}$, which accounts for $\pi_4$, recalling that the standard $L$-function of $g$ is (a translate of) its symmetric square $L$-function. Finally, Ikeda and Yamana use Proposition \ref{nonzero} to show that $g_3=3$, and prove that $\Theta^{(3)}(v_3)=\mathcal{F}_{I^{(4)}(g),g}$. (Then we may use Proposition \ref{IkMiy}(2) for the global Arthur parameter.)
\vskip5pt
\begin{center}
\begin{tabular}{|c|c|c|c|}\hline$\mathbf{i}$ & $\lambda_i\left(T_{(\sqrt{2})}\right)$ & degree & Global Arthur parameters \\\hline
 $\mathbf{1}$ & $135$ & $0$ & $[1]\oplus [7]$\\$\mathbf{2}$ & $-30$ & $4$ & $\Delta_3[4]$\\$\mathbf{3}$ & $-8$ & $3$ & $\Sym^2\Delta_3\oplus\Delta_3[2]\oplus [1]$\\$\mathbf{4}$ & $58$ & $1$ & $\Sym^2\Delta_3\oplus [5]$\\$\mathbf{5}$ & $33+3\sqrt{73}$ & $2$ & $\Delta_5^{(2)}[2]\oplus [3]$\\$\mathbf{6}$ & $33-3\sqrt{73}$ & $2$ & "\\\hline
 \end{tabular}
 \end{center}

\section{Even unimodular $8$-dimensional quadratic forms over $\Q(\sqrt{5})$}
Maass \cite{Ma} showed that if $E=\Q(\sqrt{5})$ (again $H=1$) then the genus of even unimodular $\OO_E$-lattices contains a single class if $N=4$ (in which case there will be a single $v_1=(1)$, $\pi_1$ of global Arthur parameter $[1]\oplus [3]$), and $2$ classes if $N=8$. In this latter case, we computed the matrices representing the neighbour operators $T_{(\sqrt{5})}$ and $T_{(2)}$ to be $\begin{pmatrix} 12456 & 7200\\12096 & 7560\end{pmatrix}$ and $\begin{pmatrix} 3650 & 1875\\3150 & 2175\end{pmatrix}$, respectively. For this, and similar computations referred to in later sections, we used Magma code written by M. Kirschmer, available at \url{http://www.math.rwth-aachen.de/~Markus.Kirschmer/}. The eigenvalues are in the table below. One eigenvector is $v_1={}^t(1,1)$, with $\pi_1$ of global Arthur parameter $[1]\oplus [7]$. Note that the computed $19656$ matches $5^{(8/2)-1}\tr(\diag(5^3,5^2,5,1,5^{-3},5^{-2},5^{-1},1))=5^3+\frac{5^7-1}{5-1}$. The other eigenvector is ${}^t(-25,42)$. Using Magma again, the space $S_6(\SL_2(\OO_E))$ is spanned by a single form $f$, on which the eigenvalues of the (Hilbert modular) Hecke operators $T_{(\sqrt{5})}$ and $T_{(2)}$ are $-90$ and $20$, respectively. Let $F=I^{(2)}(f)\in S_4(\Sp_2(\OO_E))$ ($\kappa=(N/2)=4, g=m=2, 2\kappa-g=6$). Then $(N/2)=4>3=m+1$, and $L(\st,F,(N/2)-m)=\zeta(2)L(f,5)L(f,4)\neq 0$, so by Proposition \ref{eigentheta3}, $F$ is in the image of $\Theta^{(2)}$, say $F=\Theta^{(2)}(v_i)$. It follows from Proposition \ref{eigentheta}(2) that $\pi_i$ has global Arthur parameter $\Delta_5[2]\oplus [1]\oplus [3]$, and the computed Hecke eigenvalue shows that it can only be $\pi_2$.
Indeed, if $t_{(\sqrt{5})}(\Delta_5)=\diag(\alpha,\alpha^{-1})$ (so $5^{5/2}(\alpha+\alpha^{-1})=-90$), then
$$5^3((\alpha+\alpha^{-1})(5^{-1/2}+5^{1/2})+1+(5^{-1}+1+5))=(-90)(1+5)+5^2(1+5+5^2)+5^3=360.$$
We could reach the same conclusions using eigenvalues of $T_{(2)}$ instead of $T_{(\sqrt{5})}$.
\vskip5pt
\begin{center}
\begin{tabular}{|c|c|c|c|c|}\hline$\mathbf{i}$ & $\lambda_i\left(T_{(\sqrt{5})}\right)$ & $\lambda_i\left(T_{(2)}\right)$ & degree & Global Arthur parameters \\\hline
 $\mathbf{1}$ & $19656$ & $5525$ & $0$ & $[1]\oplus [7]$\\$\mathbf{2}$ & $360$ & $500$ & $2$ & $\Delta_5[2]\oplus [1]\oplus [3]$\\\hline
 \end{tabular}
 \end{center}
\vskip5pt
 Note that ${}^t(0,1)=\frac{1}{67}\left(25v_1+v_2\right)$. Applying the Hecke operator $T_{\pp}$, for any prime ideal $\pp$, to both sides, it follows easily that $\lambda_1(T_{\pp})\equiv\lambda_2(T_{\pp})\pmod{67}$. This is
 $$\Nm\pp^3+(1+\Nm\pp+\Nm\pp^2+\ldots +\Nm\pp^6)\equiv a_{\pp}(f)(1+\Nm\pp)+(\Nm\pp^2+\Nm\pp^3+\Nm\pp^4)+\Nm\pp^3\pmod{67},$$
 which boils down to $(\Nm\pp+1)$ times the known Eisenstein congruence $a_{\pp}(f)\equiv 1+\Nm\pp^5\pmod{67}$, the true origin of the modulus $67$ being as a divisor of the algebraic part of the Dedekind zeta value $\zeta_E(6)$. Using the factorisation $\zeta_E(s)=\zeta(s)L(s,\chi_5)$, and using Bernoulli polynomials to compute $L(1-6,\chi_5)$, one finds $\zeta_E(6)=\frac{2^3\cdot 67\cdot\pi^{12}}{3^4\cdot 5\cdot 7}$.

 Similarly in the previous section, we could have proved congruences modulo $11$ between $\lambda_1(T_{\pp})$ and all of $\lambda_2(T_{\pp}),\lambda_3(T_{\pp}),\lambda_4(T_{\pp})$, and modulo divisors of $19$ between $\lambda_1(T_{\pp})$ and $\lambda_5(T_{\pp}),\lambda_6(T_{\pp})$. These are accounted for similarly by Eisenstein congruences in weights $4$ and $6$, with $11$ dividing $\zeta_{\Q(\sqrt{2})}(4)/\pi^8$ and $19^2$ dividing $\zeta_{\Q(\sqrt{2})}(6)/\pi^{12}$.

 To justify what we have done in this section, and what we shall do in the next, we need to take care of the following lemmas.
 \begin{lem}\label{split} For $L$ even and unimodular of rank $4n$ over $\OO_E$, where $E=\Q(\sqrt{5})$, $\SO_L$ is split at all finite places.
\end{lem}
\begin{proof} One of the classes in the genus is represented by the direct sum (let's call it $L$) of $n$ copies of a lattice representing the single class of rank $4$ even, unimodular lattices. We can take this to be a maximal order in the totally definite quaternion algebra $D$ over $E$ unramified at all finite places (the icosian ring), with bilinear form $(\alpha,\beta)\mapsto \alpha\overline{\beta}+\overline{\alpha}\beta$, so quadratic form $\alpha\mapsto \alpha\overline{\alpha}$. Since $D$ has a basis $\{1,i,j,k\}$ over $E$ satisfying the same relations as the usual Hamilton quaternions, over $E$ the quadratic form on $L$ is equivalent to $\sum_{i=1}^{4n}x_i^2$. We just need to show that at all finite places $\p$, $\sum_{i=1}^{4n}x_i^2$ is equivalent over $E_{\p}$ to $\sum_{i=1}^{2n}x_i^2-\sum_{i=2n+1}^{4n}x_i^2$, which in turn is equivalent to the desired $\sum_{i=1}^{2n}x_ix_{2n+i}$.

Two forms over a $\p$-adic field are equivalent if and only if they have the same rank, discriminant (modulo squares) and Hasse-Witt invariant. For $\sum_{i=1}^{4n}x_i^2$ and $\sum_{i=1}^{2n}x_i^2-\sum_{i=2n+1}^{4n}x_i^2$, the rank and discriminant are obviously equal. For a diagonal form $\sum_{i=1}^Na_ix_i^2$, the Hasse-Witt invariant is a product of Hilbert symbols $\prod_{i<j}(a_i,a_j)_{\p}$ \cite[Chapter IV,\S 2]{Se}. Since $z^2-(x^2+y^2)=0$ and $z^2-(x^2-y^2)=0$ have non-trivial solutions $(1,0,1)$ and $(1,1,0)$ respectively in $E_{\p}$, $(1,1)_{\p}=(1,-1)_{\p}=1$. Hence the Hasse-Witt invariants of $\sum_{i=1}^{4n}x_i^2$ and $\sum_{i=1}^{2n}x_i^2-\sum_{i=2n+1}^{4n}x_i^2$ are $1$ and $(-1,-1)_{\p}^{({2n\atop 2})}$, respectively, so it suffices to show that $(-1,-1)_{\p}=1$, i.e. that $x^2+y^2+z^2=0$ has a non-trivial solution in every $E_{\p}$. This is easy for $\p$ dividing odd $p$, where we have solutions in $\Q_p$ (by the Chevalley-Warning theorem and Hensel's lemma). For $\p=(2)$, we can use Hensel's lemma in the variable $x$ to lift the mod $8$ solution $(2+\tau,1+\tau,1)$ (where $\tau^2=1+\tau$) to a solution in $E_{\p}$. Alternatively we can use $(-1,-1)_{\infty_1}=(-1,-1)_{\infty_2}=-1$ and the product formula for the Hilbert symbol.
\end{proof}
\begin{lem}\label{red} For $L$ even and unimodular of rank $4n$ over $\OO_E$, where $E=\Q(\sqrt{5})$, and for every finite $\p$, $\SO_L/\OO_{E,\p}$ is reductive.
\end{lem}
\begin{proof}
Since $L$ is unimodular, the group scheme $\SO_L$ is reductive over $\OO_{E,\p}$ for any finite prime $\p\neq (2)$. (The special fibre is the special orthogonal group of the quadratic form associated to the reduction of the Gram matrix, which is non-singular. In characteristic $2$ we have to be more careful about the distinction between bilinear forms and quadratic forms.) The question arises whether or not $\SO_L/\OO_{E,(2)}$ is reductive.

As already remarked in the proof of Lemma \ref{split}, one of the classes in the genus is represented by the direct sum (let's call it $L$) of $n$ copies of the icosian ring $R$. Following \cite[(11.5.7)]{V}, we take $\{1,i,\zeta,i\zeta\}$ as an $\OO_E$-basis for $R$, where $\zeta:=(\tau+\tau^{-1}i+j)/2$, $\tau=(1+\sqrt{5})/2$ is the golden ratio and $i,j$ are the usual Hamilton quaternions of the same names. With respect to this basis, one easily checks that the Gram matrix of the bilinear form $(\alpha,\beta)\mapsto \alpha\overline{\beta}+\overline{\alpha}\beta$ is $\begin{pmatrix} 2 & 0 & \tau & -\tau^{-1}\\0 & 2 & \tau^{-1} & \tau\\\tau & \tau^{-1} & 2 & 0\\-\tau^{-1} & \tau & 0 & 2\end{pmatrix}$, which does have determinant $1$. Using $(2)$-integral elementary row and column operations to change the $\OO_{E,(2)}$-basis of the lattice $R\otimes\OO_{E,(2)}$, one reduces the Gram matrix to an equivalent $\begin{pmatrix} 2 & \tau & 0 & 0\\\tau & 2 & 0 & 0\\0 & 0 & 4+2\tau & -1-3\tau\\0 & 0 & -1-3\tau & 4+2\tau\end{pmatrix}$, then further to
$$\begin{pmatrix} 2 & 1 & 0 & 0\\1 & 2\tau^{-2} & 0 & 0\\0 & 0 & 2(2+\tau) & 1\\0 & 0 & 1 & 2(2+\tau)/(1+3\tau)^2\end{pmatrix}.$$ In the notation of the proof of \cite[Proposition 9]{O}, both blocks are of the form $K\simeq \begin{pmatrix} 2\epsilon & 1\\1 & 2\alpha\end{pmatrix}$. Up to squares, the determinant of the first block is $4-\tau^2=-(\tau-3)$. Since $\tau-3$ is not a square in $\OO_{E,(2)}$ and since the block is ``even'', $K\simeq H(\rho)$, in O'Meara's notation. Likewise for the second block, since the overall determinant is $1$, so $R\simeq H(\rho)\oplus H(\rho)$, which is isomorphic to $H(0)\oplus H(0)$, according to the proof of \cite[Proposition 10]{O}. In other words, with respect to some $\OO_{E,(2)}$-basis of the lattice $R\otimes\OO_{E,(2)}$, the Gram matrix is
$\begin{pmatrix} 0 & 1 & 0 & 0\\1 & 0 & 0 & 0\\0 & 0 & 0 & 1\\0 & 0 & 1 & 0\end{pmatrix}$. It follows now from \cite[Proposition C.3.10]{Con} that $\SO_L$ is reductive (in fact semi-simple) over $\OO_{E,(2)}$.
\end{proof}

\section{Even unimodular $12$-dimensional quadratic forms over $\Q(\sqrt{5})$}
Costello and Hsia \cite{CH} showed that if $E=\Q(\sqrt{5})$ then the genus of even unimodular $\OO_E$-lattices contains $15$ classes if $N=12$. We have simultaneously diagonalised the neighbour operators $T_{(2)}$ and $T_{(\sqrt{5})}$, with the eigenvalues recorded in the table below. We have also produced guesses for the global Arthur parameters that recover these computed Hecke eigenvalues (and the correct $c_{\infty_1}(z), c_{\infty_2}(z)$), with the exception of three cases. To illustrate this, consider $\mathbf{i=10}$. Using Magma, the space $S_{[10,6]}(\SL_2(\OO_E))$ (i.e. $f\left(\frac{az+b}{cz+d}\right)=(c_1z_1+d_1)^{10}(c_2z_2+d_2)^6f(z)$, non-parallel weight) is one-dimensional. The associated automorphic representation $\Delta_{(9,5)}$ of $\GL_2(\A_E)$ has $$c_{\infty_1}(\Delta_{(9,5)})(z)=\diag((z/\overline{z})^{9/2},(z/\overline{z})^{-9/2}),$$  $$c_{\infty_2}(\Delta_{(9,5)})(z)=\diag((z/\overline{z})^{5/2},(z/\overline{z})^{-5/2}).$$ Exchanging embeddings, we have $\Delta_{(5,9)}$ with $$c_{\infty_1}(\Delta_{(5,9)})(z)=\diag((z/\overline{z})^{5/2},(z/\overline{z})^{-5/2}),$$
$$c_{\infty_2}(\Delta_{(5,9)})(z)=\diag((z/\overline{z})^{9/2},(z/\overline{z})^{-9/2}).$$
Now if $\pi=\Delta_{(9,5)}[2]\oplus\Delta_{(5,9)}[2]\oplus[3]\oplus[1]$ then
$$c_{\infty_1}(\pi)=\diag((z/\overline{z})^{9/2},(z/\overline{z})^{-9/2})\otimes\diag((z/\overline{z})^{1/2},(z/\overline{z})^{-1/2})$$
$$\oplus\diag((z/\overline{z})^{5/2},(z/\overline{z})^{-5/2})\otimes\diag((z/\overline{z})^{1/2},(z/\overline{z})^{-1/2})\oplus\diag((z/\overline{z})^1,(z/\overline{z})^{-1},1,1),$$ which is conjugate (in $\GL_{12}(\C)$) to the correct $$\diag((z/\overline{z})^5,\ldots,(z/\overline{z})^1,1,(z/\overline{z})^{-5},\ldots,(z/\overline{z})^{-1},1).$$ Similarly $c_{\infty_2}(\pi)$ is correct. Here is a small table of Hecke eigenvalues of the Hilbert modular forms used in this section.
\vskip10pt
\begin{center}
\begin{tabular}{|c|c|c|}\hline & $T_{(2)}$ & $T_{(\sqrt{5})}$\\\hline $\Delta_5$ & $20$ & $-90$\\$\Delta_7$ & $140$ & $150$\\$\Delta_9^{(2)}$ & $170\mp 30\sqrt{809}$ & $570\pm 60\sqrt{809}$\\$\Delta_{(9,5)},\Delta_{(5,9)}$ & $320$ & $1950$\\$\Delta_{(7,3)},\Delta_{(3,7)}$ & $-160$ & $150$\\\hline
\end{tabular}
\end{center}
\vskip10pt
Note that in general, $\Delta_{(9,5)}$ and $\Delta_{(5,9)}$ do not have the same Hecke eigenvalues, rather they are conjugate in $\Q(\sqrt{5})$.

For $\mathbf{i=10}$ and $\pp=(2)$, if $4^{9/2}(\beta+\beta^{-1})=320$, we check
$$4^5\tr(\beta\cdot 4^{1/2},\beta\cdot 4^{-1/2},\beta^{-1}\cdot 4^{1/2},\beta^{-1}\cdot 4^{-1/2},\beta\cdot 4^{1/2},\beta\cdot 4^{-1/2},\beta^{-1}\cdot 4^{1/2},\beta^{-1}\cdot 4^{-1/2},4,1,4^{-1},1)$$
$$=2(320)(1+4)+4^4(1+4+4^2)+4^5=9600,$$ so $\Delta_{(9,5)}[2]\oplus\Delta_{(5,9)}[2]\oplus[1]\oplus[3]$ would produce the same $\lambda_{10}(T_{(2)})$ as what was computed using neighbours.
\vskip5pt
\begin{center}
\begin{tabular}{|c|c|c|c|c|}\hline$\mathbf{i}$ & $\lambda_i\left(T_{(2)}\right)$ & $\lambda_i\left(T_{(\sqrt{5})}\right)$ & $g_i$ & Global Arthur parameters (conj'l)\\\hline
 $\mathbf{1}$ & $1399125$ & $12210156$ & $0$ & $[1]\oplus[11]$\\$\mathbf{2}$ & $348900$ & $2446380$ & $1$ & $\Sym^2\Delta_5\oplus[9]$\\$\mathbf{3}$ & $89250+150\sqrt{809}$ & $494820-360\sqrt{809}$ & $2$ & $\Delta_9^{(2)}[2]\oplus[1]\oplus[7]$\\$\mathbf{4}$ & $89250-150\sqrt{809}$ & $494820+360\sqrt{809}$ & $2$ & "\\$\mathbf{5}$ & $27300$ & $-351540$ & $6$ & $\Delta_5[6]$\\$\mathbf{6}$ & $24000$ & $107100$ & $3$ & $\Sym^2\Delta_5\oplus\Delta_7[2]\oplus[5]$\\$\mathbf{7}$ & $21300$ & $90900$ & $3$ & ?\\$\mathbf{8}$ & $18300$ & $45900$ & $4$ & $\Delta_7[4]\oplus[1]\oplus[3]$\\$\mathbf{9}$ & $10800$ & $27900$ & $4$ & ?\\$\mathbf{10}$ & $9600$ & $45900$ & $4$ & $\Delta_{(9,5)}[2]\oplus\Delta_{(5,9)}[2]\oplus[1]\oplus[3]$\\$\mathbf{11}$ & $8850+150\sqrt{809}$ & $12420-360\sqrt{809}$ & $4$ & $\Delta_{9}^{(2)}[2]\oplus \Delta_5[2]\oplus[1]\oplus[3]$ \\$\mathbf{12}$ & $8850-150\sqrt{809}$ & $12420+360\sqrt{809}$ & $4$ & "\\$\mathbf{13}$ & $7200$ & $-62100$ & $5$ & $\Sym^2\Delta_5\oplus\Delta_5[4]\oplus[1]$ \\$\mathbf{14}$ & $-6000$ & $17100$ & $\leq 5$ & $\Sym^2\Delta_5\oplus\Delta_{(7,3)}[2]\oplus\Delta_{(3,7)}[2]\oplus[1]$\\$\mathbf{15}$ & $900$ & $-13500$ & $\leq 5$ & ?\\\hline
 \end{tabular}
 \end{center}
\vskip10pt
 Unlike the situation in the previous two sections, it is not possible to prove all the guesses for global Arthur parameters using theta series. But we can do most of them, all but $i=\mathbf{10,14}$.
 \begin{prop}\label{rt5N12}
 The guesses for global Arthur parameters are correct in the cases $i=\mathbf{1,2,3,4,5,6,8,11,12,13}$. In these cases, also the degrees are as in the table.
 \end{prop}
 \begin{proof}

 $\mathbf{i=1}$. This is proved just as in the previous sections.

 $\mathbf{i=2}$. The space $S_6(\SL_2(\OO_E))$ is spanned by a single form $f$, met already in the previous section. Since $(N/2)=6>1+1=m+1$, and $L(\st,f,5)=L(\Sym^2 f,10)\neq 0$, Proposition \ref{eigentheta3} tells us that $f$ belongs to the image of $\Theta^{(1)}$. The Satake parameter at $(\sqrt{5})$ for the automorphic representation of $\Sp_1(\A)$ associated with $f$ is $(\alpha^2,1,\alpha^{-2})$. If $\Theta^{(1)}(v_i)=f$ (up to scalar multiples) then $t_{(\sqrt{5})}(\pi_i)=(\diag(5^4, 5^3,\ldots,1,\alpha^2,5^{-4},5^{-3},\ldots,1,\alpha^{-2}))$, by Proposition \ref{eigentheta}(2). Now
 $$5^{(12/2)-1}\tr(\diag(5^4, 5^3,\ldots,1,\alpha^2,5^{-4},5^{-3},\ldots,1,\alpha^{-2}))$$ $$=5\frac{5^9-1}{5-1}+(5^{5/2}(\alpha+\alpha^{-1}))^2-5^5 =5\frac{5^9-1}{5-1}+(-90)^2-5^5=2446380,$$ so we must have $\Theta^{(1)}(v_2)=f$. Proposition \ref{eigentheta}(2) now shows that for every $\pp$ we have
 $$t_{\pp}(\pi_2)=(\diag(\Nm\pp^4, \Nm\pp^3,\ldots,1,\alpha_{\pp}^2,\Nm\pp^{-4},\Nm\pp^{-3},\ldots,1,\alpha_{\pp}^{-2})),$$ where $a_{\pp}(f)=(\Nm\pp)^{5/2}(\alpha_{\pp}+\alpha_{\pp}^{-1})$. Thus every local Langlands parameter of $\pi_2$ at a finite prime matches that attached to the global Arthur parameter
 $\Sym^2\Delta_5\oplus[9]$. At $\infty_1$ and $\infty_2$, $c_{\infty_j}(\Sym^2\Delta_5):z\mapsto\diag((z/\overline{z})^5,1,(z/\overline{z})^{-5})$ (for $j=1,2$), and $c_{\infty_j}([9]):z\mapsto\diag((z/\overline{z})^4,\ldots,(z/\overline{z})^1,1,(z/\overline{z})^{-4},\ldots,(z/\overline{z})^{-1})$. The concatenation matches the standard
 $c_{\infty_j}(\pi_i)(z)$. Also the other element $j$ generating $W_{\R}$ with $\C^{\times}$ (with $jzj^{-1}=\overline{z}$) acts to exchange powers of $z/\overline{z}$ with opposite exponents, for both $c_{\infty_j}(\pi_i)$ and $c_{\infty_j}(\Sym^2\Delta_5)$,$c_{\infty_j}([9])$. So all the local Langlands parameters match, and the global Arthur parameter of $\pi_2$ is as stated. For the other cases we shall not give such full details of the logic.

 $\mathbf{i=3,4}$. The space $S_{10}(\SL_2(\OO_E))$ is spanned by Galois conjugate forms $f_1,f_2$, with associated cuspidal automorphic representations of $\GL_2(\A_E)$ both denoted $\Delta_9^{(2)}$. Both $I^{(2)}(f_1)$ and $I^{(2)}(f_2)\in S_6(\Sp_2(\OO_E))$ are in the image of $\Theta^{(2)}$, by Proposition \ref{eigentheta3}, since $6>2+1$ and $L(\st,I^{(2)}(f_j),4)=\zeta(4)L(f_j,9)L(f_j,8)\neq 0$. By Proposition \ref{eigentheta}(2) (and Proposition \ref{IkMiy}(1)), the corresponding $\pi_i$ have the correct $t_{\pp}$ for $\Delta_9^{(2)}[2]\oplus[1]\oplus[7]$, which also produces the correct $c_{\infty_1},c_{\infty_2}$. Checking Hecke eigenvalues, $i$ must be $\mathbf{3}$ and $\mathbf{4}$.

 $\mathbf{i=8}$. The space $S_8(\SL_2(\OO_E))$ is spanned by a single form $g$, with associated $\Delta_7$. By Proposition \ref{eigentheta3}, $I^{(4)}(g)=\Theta^{(4)}(v_i)$ for some $i$, since $6>4+1$ and $L(\st,I^{(4)}(g),2)=\zeta(2)L(g,7)L(g,6)L(g,5)L(g,4)\neq 0$. Note that although $L(g,4)$ is a central value, the sign in the functional equation is $+1$, and in fact $L(g,4)\neq 0$. (Magma produced, after about $2$ minutes, an approximation to $29$ decimal places, beginning $1.606277885$, sufficient to prove non-vanishing.)  As before, $\pi_i$ must have global Arthur parameter $\Delta_7[4]\oplus[1]\oplus[3]$, and checking against the computed Hecke eigenvalues, $i$ must be $\mathbf{8}$.

 $\mathbf{i=5}$. By Proposition \ref{ilifttheta}(1), there is some $\pi_i$ with global Arthur parameter $\Delta_5[6]$, and it can only be $i=\mathbf{5}$, since $-90\frac{5^6-1}{5-1}=-351540$. We may also check that $L(f,3)\approx 0.854944\neq 0$, so $\Theta^{(6)}(v_5)=I^{(6)}(f)$, by Proposition \ref{ilifttheta}(2).

 $\mathbf{i=6}$. We have seen already that $\Theta^{(4)}(v_8)=I^{(4)}(g)$ (with $g\in S_8(\SL_2(\OO_E))$), and $\Theta^{(1)}(v_2)=f\in S_6(\SL_2(\OO_E))$, in particular $g_8=4$ and $g_2=1$. We find that
 $(v_8,v_2\circ v_6)\neq 0$, so by Proposition \ref{nonzero} $\Theta^{(1+g_6)}(v_8)\neq 0$, so $1+g_6\geq g_8=4$, i.e. $g_6\geq 3$. But also $(v_6,v_3\circ v_2)\neq 0$, which implies that $g_6\leq g_3+g_2=2+1=3$. Hence $g_6=3$. Knowing this, Proposition \ref{nonzero} now tells us that $\langle \Theta^{(4)}(v_8)|_{\HH_1\times\HH_3},\Theta^{(1)}(v_2)\times\Theta^{(3)}(v_6)\rangle\neq 0$, i.e. $\langle I^{(4)}(g)|_{\HH_1\times\HH_3},f\times\Theta^{(3)}(v_6)\rangle\neq 0$.
 By Proposition \ref{IkMiy}(2) then $\langle \mathcal{F}_{I^{(4)}(g),f},\Theta^{(3)}(v_6)\rangle\neq 0$, so the Hecke eigenforms $\mathcal{F}_{I^{(4)}(g),f}$ and $\Theta^{(3)}(v_6)$ must have the same Hecke eigenvalues and standard $L$-function. Using $L(\st,\mathcal{F}_{I^{(4)}(g),f},s)=L(\st,f,s)L(g,s+4)L(g,s+3)$, the global Arthur parameter of the cuspidal automorphic representation of $\Sp_3(\OO_E)$ associated to $\Theta^{(3)}(v_6)$ is $\Sym^2\Delta_5\oplus\Delta_7[2]$, then using Proposition \ref{eigentheta}(2) the global Arthur parameter of $\pi_6$ is $\Sym^2\Delta_5\oplus\Delta_7[2]\oplus[5]$ (where again one checks easily that $c_{\infty_1}$ and $c_{\infty_2}$ are right).

 We may actually say something a bit stronger about the relation between $\mathcal{F}_{I^{(4)}(g),f}$ and $\Theta^{(3)}(v_6)$, now we know that $\mathcal{F}_{I^{(4)}(g),f}\neq 0$.
 Since $N/2=6>3+1=m+1$, and since $L(\st,\mathcal{F}_{I^{(4)}(g),f},(N/2)-m)=L(\st,f,3)L(g,7)L(g,6)\neq 0$, Proposition \ref{eigentheta3} tells us that $\mathcal{F}_{I^{(4)}(g),f}$ is in the image of $\Theta^{(3)}$, and (up to scalar multiple) it can only be $\Theta^{(3)}(v_6)$.

 $\mathbf{i=11,12}$. This time use $(v_5,v_3\circ v_{11})\neq 0$ and $(v_{11},v_6\circ v_2)\neq 0$ to show that $g_{11}=4$ and $\Theta^{(4)}(v_{11})$ has the same Hecke eigenvalues as $\mathcal{F}_{I^{(6)}(f),I^{(2)}f_1}$. Then since $N/2=6>4+1=m+1$ and $$L(\st,\mathcal{F}_{I^{(6)}(f),I^{(2)}f_1},(N/2)-m)=L(\st,I^{(2)}f_1,2)L(f,5)L(f,4)$$ $$=\zeta(2)L(f_1,7)L(f_1,6)L(f,5)L(f,4)\neq 0,$$ $\Theta^{(4)}(v_{11})$ and $\mathcal{F}_{I^{(6)}(f),I^{(2)}f_1}$ are actually the same up to scalar multiples. Similarly for $\mathbf{i=12}$.

  $\mathbf{i=13}$. We argue as in the previous case, using $(v_5,v_2\circ v_{13})\neq 0$ and $(v_{13},v_6\circ v_3)\neq 0$ to prove that $g_{13}=5$ and $\Theta^{(5)}(v_{13})$ is in the same Hecke eigenspace as $\mathcal{F}_{I^{(6)}(f),f}$. This proves the guess for the global Arthur parameter and shows that $\mathcal{F}_{I^{(6)}(f),f}\neq 0$. To show that $\mathcal{F}_{I^{(6)}(f),f}$ and $\Theta^{(5)}(v_{13})$ are equal up to scalar multiple, we proceed as follows, thanks to advice from Yamana. Since $N/2=m+1=6$, Proposition \ref{eigentheta3} does not apply. In other words, we are outside the ``convergent range'' for the Siegel-Weil formula. However, a theorem of Gan, Qiu and Takeda, extending Rallis's inner product formula \cite{GQT}[Theorem 11.3] applies. In their notation, $r=0, \epsilon_0=1, m=12, n=5, d(n)=6$, and the $L$-value in their condition (b) is $L(\st,\mathcal{F}_{I^{(6)}(f),f},(N/2)-m)=L(\st,f,1)L(f,5)L(f,4)L(f,3)L(f,2)$, which is non-zero as required. Regarding the condition (a), the required non-vanishing of the local zeta integrals at infinite places is pointed out by Z. Liu in \cite{LiuZ}[\S 4.3], who attributes the computation to Shimura \cite{Sh}. Hence the theta lift to $O_{12}(\A_E)$ of (the automorphic representation associated to) $\mathcal{F}_{I^{(6)}(f),f}$ is non-zero. By a theorem of Moeglin \cite{Moe}, the theta lift of this to $\Sp_5(\A_E)$ is back where we started. It follows that $\mathcal{F}_{I^{(6)}(f),f}$ is in the image of $\Theta^{(5)}$, and (up to scalar multiple) it can only be $\Theta^{(5)}(v_{13})$.

 \end{proof}
 \begin{prop}
 The rest of the degrees in the table are correct.
 \end{prop}
 \begin{proof}

 $\mathbf{i=7}$. $(v_7,v_2\circ v_3)\neq 0\implies g_7\leq g_2+g_3=1+2=3$. But $(v_{13},v_7\circ v_3)\neq 0\implies g_7\geq g_{13}-g_3=5-2=3$, hence $g_7=3$.

 $\mathbf{i=9,10}$. $(v_{9},v_6\circ v_2)\neq 0\implies g_9\leq 4$, while $(v_5,v_9\circ v_3)\neq 0\implies g_9\geq 4$, so $g_9=4$. Similarly, non-vanishing of $(v_{10},v_6\circ v_2)$ and $(v_5,v_{10}\circ v_3)$ implies that $g_{10}=4$.

  $\mathbf{i=14,15}$. $(v_{14},v_8\circ v_2)\neq 0\implies g_{14}\leq 5$ and $(v_{15},v_{9}\circ v_2)\neq 0$ implies that $g_{15}\leq 5$.
 \end{proof}

An alternative approach to proving the global Arthur parameters for $\mathbf{i=10,14}$ (or any of the others), would be to use Arthur's multiplicity formula for symplectic groups over $E$, to prove the existence of Hecke eigenforms in $S_6(\Sp_4(\OO_E))$ and $S_6(\Sp_5(\OO_E))$ whose associated automorphic representations have global Arthur parameters $\Delta_{(9,5)}[2]\oplus\Delta_{(5,9)}[2]\oplus[1]$ and $\Sym^2\Delta_5\oplus\Delta_{(7,3)}[2]\oplus\Delta_{(3,7)}[2]$, respectively, then to proceed as in the proof of Proposition \ref{rt5N12}, to show that each is in the image of the appropriate theta map. This would be the analogue of the proof in \cite[9.2.11]{CL} for the Niemeier lattices. We do not pursue this here, because we are as yet unable to prevent this method showing that the parameter $\psi=\Delta_{(7,3)}\otimes \Delta_{(3,7)}\oplus \Delta_7[2]\oplus [3]\oplus [1]$ also occurs. This is impossible, since the eigenvalue of $T_{(2)}$ would be $(-160)^2+4(140)(1+4)+4^4(1+4+4^2)+4^5=34800$, which does not match anything in the table. Here $\Delta_{(7,3)}\otimes \Delta_{(3,7)}$ comes from a representation of $\SO_{2,2}(\A_E)$ arising via tensor-product functoriality, as explained in \cite[4.14]{ChR}.

  \subsection{Congruences mod $29$ and mod $11$}
 As in the previous section, we may easily prove the following congruences, for any prime ideal $\pp$:
 \begin{enumerate}
 \item $\lambda_1(T_\pp)\equiv\lambda_2(T_\pp)\equiv\lambda_5(T_\pp)\equiv\lambda_{13}(T_\pp)\pmod{67}$;
 \item $\lambda_3(T_\pp)\equiv\lambda_{11}(T_\pp),\lambda_4(T_\pp)\equiv\lambda_{12}(T_\pp)\pmod{67}$;
 \item $\lambda_2(T_\pp)\equiv\lambda_6(T_\pp)\pmod{19}$;
 \item $\lambda_1(T_\pp)\equiv\lambda_3(T_\pp)\pmod{\qq}$ with $\qq\mid 191$ or $2161$ (similarly for $\lambda_4(T_{\pp})$);
 \item $\lambda_8(T_\pp)\equiv\lambda_{10}(T_\pp)\pmod{29}$;
 \item $\lambda_{13}(T_\pp)\equiv\lambda_{14}(T_\pp)\pmod{11}$.
 \end{enumerate}
 The first four are accounted for by congruences between cusp forms and Eisenstein series. We have already met $67\mid(\zeta_E(6)/\pi^{12})$, but also $19\mid(\zeta_E(8)/\pi^{16})$ and $191\cdot 2161\mid(\zeta_E(10)/\pi^{20})$. In fact $19^2\mid(\zeta_E(8)/\pi^{16})$, and the congruence in (3) appears to be modulo $19^2$. To explain the congruences (5) and (6) we shall need the following.
  \begin{prop}\label{JLR}
 Let $\pi_0$ be a cuspidal automorphic representation of $\GL_2(\A_E)$ ($E$ a real quadratic field) with trivial character and $\pi_{0,\infty_1}|_{\SL_2(\R)}$ and $\pi_{0,\infty_2}|_{\SL_2(\R)}$ isomorphic to the discrete series representations $D_{k_1}^+\oplus D_{k_1}^-$, $D_{k_2}^+\oplus D_{k_2}^-$ respectively, say $k_1>k_2\geq 2$. Let $\mathcal{N}_0$ be the level of $\pi_0$ and let $N=\Nm(\mathcal{N}_0)d_E^2$, where $d_E$ is the discriminant. Then there is a Siegel cusp form $F$ of genus $2$, weight $\Sym^j(\C^2)\otimes\det^{\kappa}$, with $(j,\kappa)=(k_2-2,2+\frac{k_1-k_2}{2})$, and paramodular level $N$, such that $L(\Spin,F,s)=L(\pi_0,s)$.
 \end{prop}
 \begin{proof}
 This is a mild generalisation of part of a theorem of Johnson-Leung and Roberts \cite[Main Theorem]{JR}, which is the case $k_2=2, k_1=2n+2$. It is likewise an application of a theorem of Roberts \cite[Theorem 8.6, Introduction]{Ro}. The analysis at finite places (leading to paramodular level) is exactly as in \cite{JR}. The only difference is at archimedean places. To make the generalisation, we simply observe that the $L$-packet $\Pi(\phi(\pi_{0,\infty}))$ (in the notation of \cite[\S 3]{JR}) contains the discrete series representation of $\GSp_2(\R)$ denoted $\pi_{\lambda}[c]$ in \cite[p. 207]{Mo}, with $c=0$ and Harish-Chandra parameter $\lambda=(\lambda_1,\lambda_2)=(\frac{k_1+k_2-2}{2},\frac{k_1-k_2}{2})$. The Blattner parameter is $(\Lambda_1,\Lambda_2)=(\lambda_1,\lambda_2)+(1,2)=(\frac{k_1+k_2}{2},2+\frac{k_1-k_2}{2})$. This is $(j+\kappa,\kappa)$, where the lowest $K_{\infty}$-type is $\mathrm{Sym}^j(\C^2)\otimes \mathrm{det}^{\kappa}$, so we recover $(j,\kappa)=(k_2-2,2+\frac{k_1-k_2}{2})$.
 \end{proof}
 Note that the case $k_1=k_2$ (which requires a limit of discrete series representation with $\lambda_2=0$) appears in the proof of \cite[Theorem 3.1]{DSp}.

 \subsubsection{$\mathbf{\lambda_8(T_\pp)\equiv\lambda_{10}(T_\pp)\pmod{29}}$} Recall that the putative Arthur parameters for $\mathbf{i=8}$ and $\mathbf{i=10}$ are $\Delta_7[4]\oplus[1]\oplus[3]$ and $\Delta_{(9,5)}[2]\oplus\Delta_{(5,9)}[2]\oplus[1]\oplus[3]$, respectively.

Before explaining the congruence in question, first we consider a related congruence. We apply the above proposition with $E=\Q(\sqrt{5})$, $\pi_0=\Delta_{(9,5)}$, $\mathcal{N}=(1)$, so we get $F$ of weight $(j,\kappa)=(4,4)$ and paramodular level $5^2$. Note that $L(\Spin,F,s)$ has rational coefficients in its Dirichlet series. For primes $p\neq 5$, let $\lambda_F(p)$ be the Hecke eigenvalue for $T(p)$ (associated to $\diag(1,1,p,p)$) on $F$.
 Let $g_1, g_2$ be the conjugate pair of eigenforms spanning $S_8(\Gamma_0(5),\chi_5)$, where $$g_1=q+2\sqrt{-29}q^2+6\sqrt{-29}q^3+12q^4+(75-50\sqrt{-29})q^5+\ldots.$$ The first thing we notice of course is that the coefficient field $\Q(\sqrt{-29})$ is ramified at $29$, the prime in question. Let $\qq=(\sqrt{-29})$.
 There appears to be a congruence, for all primes $p\neq5$:
 $$\lambda_F(p)\equiv a_{g_1}(p)(1+p^2)\pmod{\qq}.$$
 For primes $p\neq 5$, since $T_p$ and $\langle p\rangle T_p$ are adjoints for the Petersson inner product, $a_p(g_1)$ is real (hence rational) or purely imaginary (hence a multiple of $\sqrt{-29}$) according as $\chi_5(p)=1$ or $-1$ respectively. When $\chi_5(p)=-1$, $\lambda_F(p)=0$ and $a_{g_1}(p)$ is a multiple of $\qq$, so the congruence holds for these $p$.
 Here is a table of what happens for the first few split primes. Note that $\lambda_F(p)=a_h(\pp)+a_h(\overline{\pp})=\tr_{E/\Q}(a_h(\pp))$, where $h$ spans $S_{[10,6]}(\SL_2(\OO_E))$, and $\pp\mid(p)$ in $E$.
 \vskip5pt
 \begin{center}
\begin{tabular}{|c|c|c|c|}\hline $p$ & $\lambda_F(p)$ & $a_{g_1}(p)$ & $\lambda_F(p)-a_{g_1}(p)(1+p^2)$\\\hline $11$ & $2184$ & $-6828$ & $29\cdot 28800$\\$19$ & $-133640$ & $6860$ & $29\cdot(-90240)$\\$29$ & $2170140$ & $25590$ & $29\cdot(-668160)$\\$31$ & $-630656$ & $82112$ & $29\cdot(-2768640)$ \\\hline
\end{tabular}
\end{center}
\vskip5pt

Observe that $29\nmid a_{g_1}(29)$, so $g_1$ is ``ordinary'' at $\qq$.
Let us now consider a non-experimental reason to believe the congruence. The right hand side of the congruence is $a_p(g_1)(1+p^{\kappa-2})$, which would be the eigenvalue of $T(p)$ on a vector-valued Klingen-Eisenstein series of weight $\Sym^j(\C^2)\otimes\det^{\kappa}$ (with $(j,\kappa)=(4,4), j+\kappa=k=8$) attached to $g_1$. The scalar-valued Klingen-Eisenstein series of paramodular level is dealt with in \cite{SS}, and the vector-valued case could be done similarly. In particular, the analysis at finite places would be the same, and we would be looking at something of paramodular level $5^2$, just like $F$. So our congruence looks like one between a cusp form and a Klingen-Eisenstein series. This is not quite so, because the convergence condition $\kappa>n+r+1=2+1+1=4$ does not hold. Nonetheless, it would be an ``Eisenstein'' congruence, between a cuspidal automorphic representation of $\GSp_2(\A)$ and an automorphic representation of $\GSp_2(\A)$ induced from the Klingen parabolic subgroup. Conjecture 4.2 of \cite{BD} is a very general conjecture on the existence of Eisenstein congruences. The case of $\GSp_2$ and its Klingen parabolic subgroup is worked out in \S 6, where the analogue of $g_1$ has trivial character, but it is easy to see that the condition under which the conjecture would predict our congruence is that $q>2(j+\kappa)$ (i.e. $29>16$) and $$\ord_{\qq}\left(\frac{L_{\{5\}}(\ad^0(g_1),3)}{\Omega}\right)>0,$$ where the adjoint $L$-function $L(\ad^0(g_1),s)$ is also $L(\Sym^2g_1,s+k-1,\chi_5)$, with $k=8$, and the subscript $\{5\}$ denotes omission of the Euler factor $(1-5^{-s})^{-1}$ at $5$. Here $\Omega$ is a Deligne period normalised as in \cite[\S 4]{BD}, and $3=1+s$ with $s=\kappa-2=2$ (which satisfies the condition $s>1$ from \cite{BD}). Note that Conjecture 4.2 of \cite{BD} only predicts a cuspidal automorphic representation, of the appropriate infinitesimal character and unramified away from $5$, and does not specify the paramodular level $5^2$ (for $F$).

The relation between the Deligne period and the Petersson norm is (up to divisors of $5(k!)$)
$$\Omega=\pi^{13}(g_1,g_1)\eta_{g_1}^{-1},$$ where $\eta_{g_1}$ is a certain congruence ideal. This employs ideas of Hida, as in \cite[\S 3]{Du}. By \cite[Proposition 2.2]{Du}, $\ord_{\qq}(\eta_{g_1})=1$. For us, $\ord_{\qq}(\eta_{g_1})>0$ would suffice, and this may appear to follow from the obvious congruence of $q$-expansions $g_1\equiv g_2\pmod{\qq}$, but note that the definition of $\eta_{g_1}$ is in terms of congruences between cohomology classes rather than $q$-expansions. Anyway, it follows that the condition $\ord_{\qq}\left(\frac{L_{\{5\}}(\ad^0(g_1),3)}{\Omega}\right)>0$ is equivalent to the integrality at $\qq$ of $\frac{L_{\{5\}}(\ad^0(g_1),3)}{\pi^{13}(g_1,g_1)}$. A theorem of Katsurada \cite[Corollary 4.3]{Ka}, which depends on $\chi_D$ being an even character, provides a way of computing this number precisely. Note that Katsurada's $L(g_1,s,\chi_D)$ is our $L_{\{5\}}(\ad^0(g_1),s)$, with the Euler factor $(1-5^{-s})^{-1}$ at $5$ already missing. Also his Petersson norm is ours divided by the volume of a fundamental domain for $\Gamma_0(5)$, which is $(\pi/3)5(1+(1/5))=2\pi$.

In Katsurada's case (c-1), substituting $m=1$ gives us a linear equation for the unknowns $\frac{\overline{c}\,L_{\{5\}}(\ad^0(g_1),3)}{\pi^{13}(g_1,g_1)}$ and $\frac{c\,L_{\{5\}}(\ad^0(g_2),3)}{\pi^{13}(g_2,g_2)}$, with coefficients the same simple multiple of $a_1(g_1)=a_1(g_2)=1$. Here $c$ is a complex number of absolute value $1$ such that $g_1| W_5=cg_2$, where $W_5$ is an Atkin-Lehner operator. First observe that $L_{\{5\}}(\ad^0(g_1),3)=L_{\{5\}}(\ad^0(g_2),3)$, since $g_2$ and $g_1$ are related by twist. Also $(g_1,g_1)=(g_2,g_2)$ since the Fourier coefficients of $g_2$ are obtained from those of $g_1$ by complex conjugation (or using the relation between $(g,g)$ and $L(\ad^0(g),1)$ \cite[Theorem 5.1]{Hi2}). Thus we actually have a linear equation in the single unknown $\frac{L_{\{5\}}(\ad^0(g_1),3)}{\pi^{13}(g_1,g_1)}$. We must check that it is non-trivial, i.e. that $\overline{c}\neq -c$. We have $c=w_{\infty}w_5$, with $w_{\infty}=(-1)^{k/2}$. By local-global compatibility \cite{Ca}, $w_5$ may be determined from a $2$-dimensional representation of the Weil group $\mathcal{W}_5$, which according to a theorem of Langlands and Carayol \cite[Theorem 4.2.7 (3)(a)]{Hi} is diagonal, so $w_5$ may be written as a product of local constants for two characters, whose product is a power of the cyclotomic character, and using Tate's local functional equation we find this product has to be $\pm 1$, in particular $\overline{c}=c$, so the linear equation for $\frac{L_{\{5\}}(\ad^0(g_1),3)}{\pi^{13}(g_1,g_1)}$ is non-trivial.  The ``right-hand-side'' of the linear equation, which comes from Fourier coefficients of an Eisenstein series of genus $2$, is very complicated, and would be tedious to compute exactly, but it is not too difficult to see at least that the solution to the equation will be integral at $\qq$, as required.

Now the congruence between $\Delta_7[4]\oplus[1]\oplus[3]$ and $\Delta_{(9,5)}[2]\oplus\Delta_{(5,9)}[2]\oplus[1]\oplus[3]$ can be accounted for by the apparent congruence we have just been discussing. This is because $\Delta_7$ is the base-change to $E$ of the cuspidal automorphic representation of $\GL_2(\A)$ attached to $g_1$ (or equally to $g_2$, which is the newform associated to the twist by $\chi_5$ of $g_1$), and because the Satake parameters of $F$ are ``induced'' from those of $\Delta_{(9,5)}$ (or equally of $\Delta_{(5,9)}$), as in Proposition \ref{JLR}. For example, at a factor $\pp$ of a split prime $p$, the congruence between $\Delta_7[4]\oplus[1]\oplus[3]$ and $\Delta_{(9,5)}[2]\oplus\Delta_{(5,9)}[2]\oplus[1]\oplus[3]$ would give
$$a_{g_1}(p)(1+p+p^2+p^3)+p^5+(p^4+p^5+p^6)$$ $$\equiv a_h(\pp)(1+p)+a_h(\overline{\pp})(1+p)+p^5+(p^4+p^5+p^6)\pmod{\qq}.$$
This is $$a_{g_1}(p)(1+p^2)(1+p)\equiv \lambda_F(p)(1+p)\pmod{\qq},$$ which is just $(1+p)$ times the Klingen-Eisenstein congruence.
\begin{remar} The $2$-dimensional mod $\qq$ representation of $\Gal(\Qbar/\Q)$ attached to $g_1$ is ``dihedral'', in particular its restriction to $\Gal(\Qbar/E)$ is reducible, cf. \cite[Proposition 1.2(2)]{Du}. The congruence would imply that the $2$-dimensional mod $\qq$ representation of $\Gal(\Qbar/E)$ attached to $\Delta_{(9,5)}$ is likewise reducible.
In fact, it appears to be the case that if $\alpha$ is a totally positive generator of any prime ideal $\pp$ in $\OO_E$ (even $\pp=(\sqrt{5})$), with algebraic conjugate $\overline{\alpha}$, and $\qq'=(29,\sqrt{5}-11)$, then
$$a_h(\pp)\equiv \overline{\alpha}^7+\alpha^9\overline{\alpha}^2\equiv\overline{\alpha}^7+\alpha^7\Nm\pp^2 \pmod{\qq'}.$$
This is independent of the choice of $\alpha$, since if $\epsilon^+$ is a totally positive unit of $\OO_E$ then $(\epsilon^+)^7\equiv 1\pmod{\qq'}$, which is what leads to the dihedral congruence, cf. \cite[Proposition 1.2(4)]{Du}. Without proving the global Arthur parameter for $\mathbf{i=10}$, we have not actually proved this congruence for $a_h(\pp)$.
It should be compared (for split $p$) with the congruence
$$a_{g_1}(p)\equiv \overline{\alpha}^7+\alpha^7 \pmod{\qq'}.$$ We can see how $a_{g_1}(p)(1+p^2)$ gets to be the same as $a_h(\pp)+\overline{a_h(\pp)}\pmod{\qq'}$, how one-dimensional composition factors get rearranged.
\end{remar}
\begin{remar} The same argument as above shows that also $\ord_{\qq}\left(\frac{L_{\{5\}}(\ad^0(g_1),5)}{\Omega}\right)>0$ and $\ord_{\qq}\left(\frac{L_{\{5\}}(\ad^0(g_1),7)}{\Omega}\right)>0$, so we would expect to observe also congruences of Klingen-Eisenstein type for $g_1$ with $(j,\kappa)=(2,6)$ and $(0,8)$, i.e. $(j+2\kappa-3,j+1)=(11,3)$ and $(13,1)$, and indeed we do.
We find that $\dim(S_{[12,4]}(\SL_2(\OO_E)))=1$, and for the associated $F$ of weight $(j,\kappa)=(2,6)$ and paramodular level $5^2$,
\vskip5pt
\begin{center}
\begin{tabular}{|c|c|c|c|}\hline $p$ & $\lambda_F(p)$ & $a_{g_1}(p)$ & $\lambda_F(p)-a_{g_1}(p)(1+p^4)$\\\hline $11$ & $-795576$ & $-6828$ & $29\cdot 3420000 $\\$19$ & $21628600$ & $6860$ & $29\cdot(-30082080)$\\$29$ & $-36938100$ & $25590$ & $29\cdot(-625389120)$\\$31$ & $92822464$ & $82112$ & $29\cdot(-2611704000)$ \\\hline
\end{tabular}
\end{center}
\vskip5pt
Moreover, if now $h$ denotes a generator of $S_{[12,4]}(\SL_2(\OO_E))$, then it appears that
$$a_h(\pp)\equiv \overline{\alpha}^7+\alpha^{11}\overline{\alpha}^4\equiv \overline{\alpha}^7+\alpha^7\Nm\pp^4\pmod{\qq'}.$$

Similarly, $\dim(S_{[14,2]}(\SL_2(\OO_E)))=1$, and for the associated $F$ of weight $(j,\kappa)=(0,8)$ and paramodular level $5^2$ we find
\vskip5pt
\begin{center}
\begin{tabular}{|c|c|c|c|}\hline $p$ & $\lambda_F(p)$ & $a_{g_1}(p)$ & $\lambda_F(p)-a_{g_1}(p)(1+p^6)$\\\hline $11$ & $8606664$ & $-6828$ & $29\cdot 417408000 $\\$19$ & $333407800$ & $6860$ & $29\cdot(-11117287680)$\\$29$ & $-7660887300$ & $25590$ & $29\cdot(-15754334169600)$\\$31$ & $-200383616$ & $82112$ & $29\cdot(-2512927680000)$ \\\hline
\end{tabular}
\end{center}
\vskip5pt
and if now $h$ denotes a generator of $S_{[14,2]}(\SL_2(\OO_E))$, then it appears that
$$a_h(\pp)\equiv \overline{\alpha}^7+\alpha^{13}\overline{\alpha}^6\equiv \overline{\alpha}^7+\alpha^7\Nm\pp^6\pmod{\qq'}.$$
\end{remar}

\subsubsection{$\mathbf{\lambda_{13}(T_\pp)\equiv\lambda_{14}(T_\pp)\pmod{11}}$} Recall that the putative Arthur parameters for $\mathbf{i=13}$ and $\mathbf{i=14}$ are $\Sym^2\Delta_5\oplus\Delta_5[4]\oplus[1]$ and $\Sym^2\Delta_5\oplus\Delta_{(7,3)}[2]\oplus\Delta_{(3,7)}[2]\oplus[1]$, respectively. We apply Proposition \ref{JLR} with $E=\Q(\sqrt{5})$, $\pi_0=\Delta_{(7,3)}$, $\mathcal{N}=(1)$, so we get $F$ of weight $(j,\kappa)=(2,4)$ and paramodular level $5^2$. For primes $p\neq 5$, let $\lambda_F(p)$ be the Hecke eigenvalue for $T(p)$ on $F$.
 Let $f_1, f_2$ be the conjugate pair of eigenforms spanning $S_6(\Gamma_0(5),\chi_5)$, where $$f_1=q-2\sqrt{-11}q^2+6\sqrt{-11}q^3-12q^4+(-45-10\sqrt{-11})q^5+\ldots.$$
 Noting the appearance of $\sqrt{-11}$, we may now proceed very much as in the other case. In particular, $\Delta_5$ is the base-change to $E$ of the cuspidal automorphic representation of $\GL_2(\A)$ attached to $f_1$ (or equally to $f_2$). Further, if now $h$ denotes a generator of $\dim(S_{[8,4]}(\SL_2(\OO_E))$ and $\qq'=(11,\sqrt{5}-4)$, then it appears that
$$a_h(\pp)\equiv \overline{\alpha}^5+\alpha^7\overline{\alpha}^2\equiv \overline{\alpha}^5+\alpha^5\Nm\pp^2\pmod{\qq'}.$$
As in all the above cases, the coefficient field of $h$ is $E$, and $a_h(\overline{\pp})$ is the algebraic conjugate of $a_h(\pp)$.

Since $\dim(S_{[10,2]}(\SL_2(\OO_E)))=0$, we cannot find a congruence for $(j,\kappa)=(0,6)$ in the same manner. We may explain this as follows. Suppose there is a congruence $\lambda_F(p)\equiv a_{f_1}(p)(1+p^{\kappa-2})\pmod{\qq}$, where $F$ is a genus-$2$ cuspidal Hecke eigenform of weight $(j,\kappa)$ and level trivial away from $5$, with irreducible $4$-dimensional $\qq$-adic Galois representation $\rho_{F,\qq}$. We expect $\ord_{\qq}\left(\frac{L_{\{5\}}(\ad^0(f_1),\kappa-1)}{\Omega}\right)> 0$ by the Bloch-Kato conjecture, because an adaptation of a well-known construction of Ribet produces a non-trivial extension of $\rhobar_{f_1,\qq}$ ($2$-dimensional mod $\qq$ Galois representation attached to $f_1$) by $\rhobar_{f_1,\qq}(2-\kappa)$ (Tate twist) inside the residual representation $\rhobar_{F,\qq}$, and a non-zero class in $H^1(\Q,\ad^0\rhobar_{f_1,\qq}(2-\kappa))$. This satisfies the Bloch-Kato local conditions away from $p=5$, so contributes to the numerator of the conjectural formula for $\frac{L_{\{5\}}(\ad^0(f_1),\kappa-1)}{\Omega}$. Now if the congruence arises in the special way described above, via a congruence for a non-parallel weight Hilbert modular form, because that form has level $1$ it is not difficult to show (using inflation-restriction) that the class also satisfies the local condition at $5$, so we should in fact see $\ord_{\qq}\left(\frac{L(\ad^0(f_1),\kappa-1)}{\Omega}\right)> 0$, for the complete $L$-value with no missing Euler factor. However, what is special about this example is that $5^5\equiv 1\pmod{11}$, so that $\ord_{\qq}((1-5^{-5})^{-1})=-2$, hence when the Euler factor is put back in, $$\ord_{\qq}\left(\frac{L(\ad^0(f_1),\kappa-1)}{\Omega}\right)<\ord_{\qq}\left(\frac{L_{\{5\}}(\ad^0(f_1),\kappa-1)}{\Omega}\right),$$ making it seem unlikely that $\ord_{\qq}\left(\frac{L(\ad^0(f_1),\kappa-1)}{\Omega}\right)>0$. Thus, though we may still expect the Klingen-Eisenstein congruence to happen, we shouldn't expect it to arise from a Johnson-Leung-Roberts lift of a non-parallel weight Hilbert modular form satisfying the type of congruence encountered in the other examples.

\subsubsection{Examples with $E=\Q(\sqrt{2})$} To reinforce what we have found, we consider two more examples. If $E=\Q(\sqrt{2})$ then $D=8$. The space $S_4(\Gamma_0(8),\chi_8)$ is spanned by a conjugate pair of eigenforms, one of which is $$q+(-1-\sqrt{-7})q^2+2\sqrt{-7}q^3+(-6+2\sqrt{-7})q^4-4\sqrt{-7}q^5+(14-2\sqrt{-7})q^6-8q^7+\ldots,$$ for which $\qq=(\sqrt{-7})$ is a dihedral congruence prime. Letting $(j,\kappa)=(0,4)$, so $[j+2\kappa-2,j+2]=[6,2]$, we might expect a congruence involving a Hecke eigenform in $S_{[6,2]}(\SL_2(\OO_E))$, but $\dim(S_{[6,2]}(\SL_2(\OO_E)))=0$. As in the previous paragraph, we can explain this failure by $\ord_{\qq}((1-2^{-(\kappa-1)})^{-1})<0$, since $2^3\equiv 1\pmod{7}$.

On the other hand, the space $S_{14}(\Gamma_0(8),\chi_8)$ is spanned by a conjugate pair of eigenforms, one of which is $q+(-56-8\sqrt{-79})q^2+258\sqrt{-79}q^3+\ldots,$ for which $\qq=(\sqrt{-79})$ is a dihedral congruence prime. Letting $(j,\kappa)=(2,14)$, so $[j+2\kappa-2,j+2]=[26,2]$, we find that $S_{[26,2]}(\SL_2(\OO_E))$ is spanned by a pair of Hecke eigenforms, with coefficient field $E(\sqrt{11713})$, conjugate over $E$. Letting $h$ be one of them, and $\qq'=(\sqrt{2}-9,\sqrt{11713}-10)$, a divisor of $79$, it does appear that for $\alpha$ any totally positive generator of a prime ideal $\pp$, with $\Gal(E/\Q)$-conjugate $\overline{\alpha}$,
$$a_h(\pp)\equiv \overline{\alpha}^{13}+\alpha^{25}\overline{\alpha}^{12}\equiv \overline{\alpha}^{13}+\alpha^{13}\Nm\pp^{12}\pmod{\qq'}.$$
Note that $2^{13}\not\equiv 1\pmod{79}$.

All of this seems to support the following conjecture. (We have to introduce the character $\psi$ because we no longer assume that $E$ has narrow class number $1$.)
Let $g\in S_k(\Gamma_0(D),\chi_D)$ be a normalised Hecke eigenform, where $D>0$ is the discriminant of a real quadratic field $E=\Q(\sqrt{D})$, with associated character $\chi_D$. Let $g^c$ be the normalised Hecke eigenform obtained from $g$ by complex-conjugating the Fourier coefficients. Suppose that $g\equiv g^c\pmod{\qq}$, where $\qq\mid q$, with $q>2k$ and $q\nmid D$, is a prime divisor of the coefficient field $K_g$, ramified in $K_g/K_g^+$, where $K_g^+$ is the totally real subfield of the CM field $K_g$. Suppose also that $g$ is ordinary at $\qq$ and that the residual representation $\rhobar_{g,\qq}$ of $\Gal(\Qbar/\Q)$ is absolutely irreducible. Necessarily $\rhobar_{g,\qq}$ is induced from a character of $\Gal(\Qbar/E)$ associated by class field theory with $\psi:\A_E^{\times}/E^{\times}\rightarrow \FF_q^{\times}$, a finite-order character of conductor $\mathcal{Q}+\infty_1$ such that $\psi(a)\equiv a^{1-k}\pmod{\mathcal{Q}}$ for $a\in \OO_{\mathcal{Q}}^{\times}$, where $(q)=\mathcal{Q}\overline{\mathcal{Q}}$ in $\OO_E$. It is also induced from $\overline{\psi}$, the $\Gal(E/\Q)$ conjugate, of conductor $\overline{\mathcal{Q}}+\infty_2$. (See \cite[Theorems 2.1, 2.11]{BG} and their proofs for more on this.)
\begin{conj}\label{hilbcong}
If $k=j+\kappa$ with $j\geq 0$ even and $\kappa\geq 4$, and if for all primes $p\mid D$, $p^{\kappa-1}\not\equiv 1\pmod{q}$, then there exists a cuspidal eigenform $h\in S_{[j+2\kappa-2,j+2]}(\GL_2(\A_E),\prod\GL_2(\OO_{\pp}))$ and a prime divisor $\qq'\mid q$ in $K_h$ such that for any prime $\pp\nmid q$ of $\OO_E$,
$$a_h(\pp)\equiv \overline{\psi}(\pp)+\psi(\pp)\Nm\pp^{\kappa-2}\pmod{\qq'}.$$
\end{conj}
For comparison, note that when $\kappa=2$ we have the base change (of $g$) $$h\in S_{[k,k]}(\GL_2(\A_E),\prod\GL_2(\OO_{\pp})),$$ satisfying
$$a_h(\pp)\equiv \overline{\psi}(\pp)+\psi(\pp)\pmod{\qq'}.$$

\section{Even unimodular quadratic forms over $\Q(\sqrt{3})$}
According to Hung \cite{Hu}, if $E=\Q(\sqrt{3})$ then there is a unique genus of even unimodular $\OO_E$-lattices for each even $N\geq 2$. He showed that it contains $1$ class when $N=2$, $2$ classes when $N=4$, $6$ classes when $N=6$ and $31$ classes when $N=8$. We have simultaneously diagonalised certain neighbour operators $T_\pp$ and recorded the eigenvalues later in this section. We have also produced guesses for the global Arthur parameters that recover these computed Hecke eigenvalues (and the correct $c_{\infty_1}(z), c_{\infty_2}(z)$), with the exception of three cases when $N=8$.
Things are different now, because although the class number is $1$, the narrow class number is $2$. The ray class field of conductor $\infty_1+\infty_2$ is $H=\Q(\sqrt{3},i)=\Q(\zeta_{12})$. Let $\chi:\GL_1(\A_E)\rightarrow\C^{\times}$ be the ray class character of conductor $\infty_1+\infty_2$. It takes the value $1$ on inert primes and totally positive split primes, $-1$ on the rest. It is now possible to have non-zero forms of odd weights. The central character of the associated automorphic representation of $\GL_2(\A_E)$ is then $\chi$.

Let $\Delta_3$ be the automorphic representation of $\GL_2(\A_E)$ attached to one of the Galois conjugate pair of Hecke eigenforms spanning $S_{[4,4]}(\GL_2(\A_E),\prod\GL_2(\OO_{\pp}))$. Its Galois conjugate is $\chi\otimes\Delta_3$.

Let $\Delta_5^{(4)}$ be any of the four Galois conjugate Hecke eigenforms spanning \newline $S_{[6,6]}(\GL_2(\A_E),\prod\GL_2(\OO_{\pp}))$ (so this symbol means four different things on different lines of the table).

There are three non-identity elements of $\Gal(H/\Q)$, i.e.
$$\tau:\sqrt{3}\mapsto -\sqrt{3}, i\mapsto i,\,\,\,\,\sigma:\sqrt{3}\mapsto\sqrt{3}, i\mapsto -i,\,\,\,\, \sigma\tau:\sqrt{3}\mapsto -\sqrt{3}, i\mapsto -i.$$
The space $S_{[7,7]}(\GL_2(\A_E),\prod\GL_2(\OO_{\pp}))$ is $3$-dimensional. One of the spanning Hecke eigenforms is CM, associated to a Hecke character of $H$ with $\infty$-type $z\mapsto z^6\sigma\tau(z)^6$. Let $\Delta_6$ be the associated cuspidal automorphic representation of $\GL_2(\A_E)$. We have
$$c_{\infty_1}(\Delta_6)(z)=c_{\infty_2}(\Delta_6)(z)=\diag((z/\overline{z})^{6/2},(z/\overline{z})^{-6/2}),$$
(coming from the $z^6$ and $\sigma\tau(z)^6$ factors respectively). Alternatively, $\Delta_6$ is the base-change to $E$ of the cuspidal automorphic representation of $\GL_2(\A)$ attached to the CM newform $q-27q^3+64q^4-286q^7+\ldots$ spanning $S_7(\Gamma_0(3),\chi_{-3})$.

The space $S_{[5,5]}(\GL_2(\A_E),\prod\GL_2(\OO_{\pp}))$ is $1$-dimensional, spanned by a CM form, associated to a Hecke character of $H$ with $\infty$-type $z\mapsto z^4\tau(z)^4$. Let $\Delta_4$ be the associated cuspidal automorphic representation of $\GL_2(\A_E)$. We have
$$c_{\infty_1}(\Delta_4)(z)=c_{\infty_2}(\Delta_4)(z)=\diag((z/\overline{z})^{4/2},(z/\overline{z})^{-4/2}),$$
(coming from the $z^4$ and $\tau(z)^4$ factors respectively). Alternatively, $\Delta_4$ is the base-change to $E$ of the cuspidal automorphic representation of $\GL_2(\A)$ attached to the CM newform $q-4q^2+16q^4-14q^5-\ldots$ spanning $S_5(\Gamma_0(4),\chi_{-4})$.

The space $S_{[6,2]}(\GL_2(\A_E),\prod\GL_2(\OO_{\pp}))$ is $1$-dimensional, spanned by a CM form, associated to a Hecke character of $H$ with $\infty$-type $z\mapsto z^5\tau(z)^2\sigma\tau(z)^3$. Let $\Delta_{(5,1)}$ be the associated cuspidal automorphic representation of $\GL_2(\A_E)$. We have
$$c_{\infty_1}(\Delta_4)(z)=\diag((z/\overline{z})^{5/2},(z/\overline{z})^{-5/2}),\,\,c_{\infty_2}(\Delta_4)(z)=\diag((z/\overline{z})^{1/2},(z/\overline{z})^{-1/2}),$$
from $\frac{z^5}{(z\overline{z})^{5/2}}=(z/\overline{z})^{5/2}$ and $\frac{z^2\overline{z}^3}{(z\overline{z})^{5/2}}=(z/\overline{z})^{-1/2}$, respectively.

Some Hecke eigenvalues:
\vskip5pt
\begin{center}
\begin{tabular}{|c|c|c|c|c|}\hline & $T_{(1+\sqrt{3})}$ & $T_{(\sqrt{3})}$ & $T_{(4+\sqrt{3})}$ & $T_{(5)}$\\\hline$\Delta_3$ & $2\sqrt{3}$ & $-4\sqrt{3}$ & $-10$ & $170$\\$\Delta_6$ & $0$ & $0$ & $506$ & $2\cdot 5^6=31250$\\$\Delta_4$ & $0$ & $0$ & $-238$ & $(2+i)^8+(2-i)^8=-1054$\\$\Delta_{(5,1)}$ & $0$ & $0$ & $350-432\sqrt{3}$ & $5^3((2+i)^4+(2-i)^4)=-1750$\\$\Delta_{(1,5)}$ & $0$ & $0$ & $350+432\sqrt{3}$ & $-1750$\\\hline
\end{tabular}
\end{center}
\vskip5pt
The zeros result from $\chi(1+\sqrt{3})=\chi(\sqrt{3})=-1$ and the CM nature of the forms. For some of the other entries, $z=\frac{3}{2}+i+\frac{\sqrt{3}}{2}i$ is an element of $H$ generating a prime ideal of norm $13$, dividing $(4+\sqrt{3})$. One finds that
$$(z\sigma\tau(z))^6=\left(\frac{5+3\sqrt{3}i}{2}\right)^6=253-1260\sqrt{3}i,\,\,\,\,(z\sigma\tau(z))^6+\sigma(z\sigma\tau(z))^6=2\times 253=506,$$
$$(z\tau(z))^4=(2+3i)^4=-119-120i,\,\,\,\,(z\tau(z))^4+\sigma(z\tau(z))^4=2\times(-119)=-238,\text{ and}$$
$$z^5\tau(z)^2\sigma\tau(z)^3=175-420i-90\sqrt{3}i-216\sqrt{3},\,\,\,\,2\times(175-216\sqrt{3})=350-432\sqrt{3}.$$
\vskip5pt
\par $\mathbf{N=2}$
\begin{center}
\begin{tabular}{|c|c|c|c|c|}\hline$\mathbf{i}$ & $\lambda_i\left(T_{(1+\sqrt{3})}\right)$ & $\lambda_i\left(T_{(\sqrt{3})}\right)$ & $\lambda_i\left(T_{(5)}\right)$  & Global Arthur parameters (conj'l)\\\hline $\mathbf{1}$ & $0$ & $0$ & $2$ & $[1]\oplus\chi$
\\\hline
\end{tabular}
\end{center}
\vskip5pt
\par $\mathbf{N=4}$
\begin{center}
\begin{tabular}{|c|c|c|c|c|}\hline$\mathbf{i}$ & $\lambda_i\left(T_{(1+\sqrt{3})}\right)$ & $\lambda_i\left(T_{(\sqrt{3})}\right)$ & $\lambda_i\left(T_{(5)}\right)$  & Global Arthur parameters (conj'l)\\\hline $\mathbf{1}$ & $9$ & $16$ & $676$ & $[1]\oplus[3]$\\$\mathbf{2}$ & $-9$ & $-16$ & $676$ & $\chi\otimes(")$
\\\hline
\end{tabular}
\end{center}
Here, $9=2+(1+2+2^2)=(1+2)^2$, $16=(1+3)^2$ and $676=(1+25)^2$.
\vskip5pt
\par $\mathbf{N=6}$
\begin{center}
\begin{tabular}{|c|c|c|c|c|}\hline$\mathbf{i}$ & $\lambda_i\left(T_{(1+\sqrt{3})}\right)$ & $\lambda_i\left(T_{(\sqrt{3})}\right)$ & $\lambda_i\left(T_{(5)}\right)$  & Global Arthur parameters (conj'l)\\\hline $\mathbf{1}$ & $27$ & $112$ & $407526$ & $[5]\oplus\chi$\\$\mathbf{2}$ & $-27$ & $-112$ & $407526$ & $\chi\otimes(")$\\$\mathbf{3}$ & $18$ & $48$ & $15846$ & $\Delta_4\oplus[3]\oplus[1]$\\$\mathbf{3}$ & $-18$ & $-48$ & $15846$ & $\chi\otimes(")$\\$\mathbf{5}$ & $6\sqrt{3}$ & $16\sqrt{3}$ & $5670$ & $\Delta_3[2]\oplus[1]\oplus\chi$\\$\mathbf{6}$ & $-6\sqrt{3}$ & $-16\sqrt{3}$ & $5670$ & $\chi\otimes(")$
\\\hline
\end{tabular}
\end{center}
Note that $27=(1+2+\ldots +2^4)-2^2$, $112=(1+3+\ldots + 3^4)-3^2$ but $407526=(1+25+\dots +25^4)+25^2$.
\vskip5pt
\par $\mathbf{N=8}$
\begin{center}
\begin{tabular}{|c|c|c|c|c|}\hline$\mathbf{i}$ & $\lambda_i\left(T_{(1+\sqrt{3})}\right)$ & $\lambda_i\left(T_{(\sqrt{3})}\right)$ & $\lambda_i\left(T_{(4+\sqrt{3})}\right)$  & Global Arthur parameters (conj'l)\\\hline
 $\mathbf{1}$ & $135$ & $1120$ & $5231240$ & $[1]\oplus [7]$\\$\mathbf{2}$ & $-135$ & $-1120$ & $5231240$ & $\chi\otimes(")$\\$\mathbf{3}$ & $54$ & $336$ & $404936$ & $\Delta_6\oplus\chi\oplus[5]$\\$\mathbf{4}$ & $-54$ & $-336$ & $404936$ & $\chi\otimes(")$\\$\mathbf{5}$ & $66$ & $384$ & $400136$ & $\Sym^2\Delta_3\oplus[5]$\\$\mathbf{6}$ & $-66$ & $-384$ & $400136$ & $\chi\otimes(")$\\$\mathbf{7}$ & $a$ (degree $4$) & $b$ (degree $4$) & $33320+672\sqrt{73}$ & $\Delta^{(4)}_5[2]\oplus[1]\oplus[3]$\\$\mathbf{8}$ & $-a$ & $-b$ & $33320+672\sqrt{73}$ & $\chi\otimes(")$\\$\mathbf{9}$ & conj. of $a$ & conj. of $b$ & $33320+672\sqrt{73}$ & $\Delta^{(4)}_5[2]\oplus[1]\oplus[3]$\\$\mathbf{10}$ & conj. of $-a$ & conj. of $-b$ & $33320+672\sqrt{73}$ & $\chi\otimes(")$\\$\mathbf{11}$ & conj. of $a$ & conj. of $b$ & $33320-672\sqrt{73}$ & $\Delta^{(4)}_5[2]\oplus[1]\oplus[3]$\\$\mathbf{12}$ & conj. of $-a$ & conj. of $-b$ & $33320-672\sqrt{73}$ & $\chi\otimes(")$\\$\mathbf{13}$ & conj. of $a$ & conj. of $b$ & $33320-672\sqrt{73}$ & $\Delta^{(4)}_5[2]\oplus[1]\oplus[3]$\\$\mathbf{14}$ & conj. of $-a$ & conj. of $-b$ & $33320-672\sqrt{73}$ & $\chi\otimes(")$\\$\mathbf{15}$ & $36$ & $144$ & $30536$ & $\Delta_6\oplus \Delta_4 \oplus [1]\oplus[3]$\\$\mathbf{16}$ & $-36$ & $-144$ & $30536$ & $\chi\otimes(")$\\$\mathbf{17}$ & $24$ & $96$ & $25736$ & $\chi\otimes\Sym^2\Delta_3\oplus\Delta_4\oplus[3]$\\$\mathbf{18}$ & $-24$ & $-96$ & $25736$ & $\chi\otimes(")$\\$\mathbf{19}$ & $0$ & $0$ & $9800$ & $\Delta_{(5,1)}[2]\oplus\Delta_{(1,5)}[2]$\\$\mathbf{20}$ & $0$ & $0$ & $6344$ & ?\\$\mathbf{21}$ & $0$ & $0$ & $5384$ & ?\\$\mathbf{22}$ & $12\sqrt{3}$ & $-48\sqrt{3}$ & $3080$ & $\Delta_6\oplus\Delta_3[2]\oplus[1]\oplus\chi$\\$\mathbf{23}$ & $-12\sqrt{3}$ & $48\sqrt{3}$ & $3080$ & $\chi\otimes(")$\\$\mathbf{24}$ & $0$ & $0$ & $1160$ & ?\\$\mathbf{25}$ & $12(\sqrt{3}+1)$ & $48(1-\sqrt{3})$ & $-1720$ & $\Sym^2\Delta_3\oplus\Delta_3[2]\oplus[1]$\\$\mathbf{26}$ & $-12(\sqrt{3}+1)$ & $-48(1-\sqrt{3})$ & $-1720$ & $\chi\otimes(")$\\$\mathbf{27}$ & $12(\sqrt{3}-1)$ & $-48(1+\sqrt{3})$ & $-1720$ & $\chi\otimes\Sym^2\Delta_3\oplus\Delta_3[2]\oplus\chi$\\$\mathbf{28}$ & $-12(\sqrt{3}-1)$ & $48(1+\sqrt{3})$ & $-1720$ & $\chi\otimes(")$\\$\mathbf{29}$ & $30\sqrt{3}$ & $-160\sqrt{3}$ & $-23800$ & $\Delta_3[4]$\\$\mathbf{30}$ & $-30\sqrt{3}$ & $160\sqrt{3}$ & $-23800$ & $\chi\otimes(")$\\$\mathbf{31}$ & $0$ & $0$ & $-39160$ & $\Delta_4[3]\oplus[1]\oplus\chi$\\
 \hline
 \end{tabular}
 \end{center}

 The numbers $a$ and $b$ appearing in the last table are roots of the polynomials $x^4-132x^2+1728$ and $x^4-960x^2+62208$ respectively. For an Arthur parameter $A$, the meaning of $\chi\otimes A$ is that each constituent $\Pi_k[d_k]$ gets replaced by $(\chi\circ\det)\otimes\Pi_k[d_k]$, where the $\det$ is of $\GL_{n_kd_k}(\A_E)$. Thus each Satake parameter $t_{\pp}(\pi_i)$ gets replaced by $\chi(\pp)t_{\pp}(\pi_i)$, so $\lambda_i(T_{\pp})$ by $\chi(\pp)\lambda_i(T_{\pp})$.  We can explain some of what is observed in the table.
 \begin{prop}\label{pairs}
 \begin{enumerate}
 \item Given an eigenvector $v_i$, there is an eigenvector $v_j$ such that $\lambda_j(T_{\pp})=\chi(\pp)\lambda_i(T_{\pp})$ for all $\pp$. In other words, if $A_i$ and $A_j$ are the associated global Arthur parameters then $A_j=\chi\otimes A_i$.
 \item $A_i=\chi\otimes A_i$ precisely for $i=19,20,21,24,31$. In particular, for these $i$, $\lambda_i(T_{\pp})=0$ whenever $\chi(\pp)=-1$.
 \end{enumerate}
 \end{prop}
 \begin{proof}
 \begin{enumerate}
 \item The $31$ classes are divided into spinor genera of sizes $18$ and $13$. When $\chi(\pp)=-1$, $\pp$-neighbours must be in different spinor genera, as may be deduced from \cite[(1.1)]{BH}, see also \cite[\S 5]{Hu}. We may resolve $v_i$ into components $a_i$ and $b_i$ supported on the classes in one spinor genus or the other. We must have $T_{\pp}(a_i)=\lambda_i(T_{\pp})b_i$ and $T_{\pp}(b_i)=\lambda_i(T_{\pp})a_i$. Letting $v_j=a_i-b_i$ then $T_{\pp}(v_j)=-\lambda_i(T_{\pp})v_j$. On the other hand, for $\chi(\pp)=1$, $\pp$-neighbours are in the same spinor genus, so $T_{\pp}(a_i)=\lambda_i(T_{\pp})a_i$, $T_{\pp}(b_i)=\lambda_i(T_{\pp})b_i$ and $T_{\pp}(v_j)=\lambda_i(T_{\pp})v_j$. Thus $v_j$ is an eigenvector with the required property.
     \item We see from the table that the eigenvalues of $T_{(4+\sqrt{3})}$ are not repeated precisely for the values of $i$ listed, so we must have $A_i=\chi\otimes A_i$ for these values of $i$. For all other values of $i$, $\lambda_i(T_{(1+\sqrt{3})})\neq 0$, so we cannot have $A_i=\chi\otimes A_i$.
     \end{enumerate}
 \end{proof}
 \begin{remar}
 When $\chi(\pp)=-1$, $T_{\pp}$ maps an $18$-dimensional subspace of $M(\C,K)$ to a $13$-dimensional subspace, with a kernel necessarily of dimension at least $5$. So $5$ was the expected number of $i$ such that $A_i=\chi\otimes A_i$. We are unable to identify conjectural Arthur parameters for three of them. The other two involve CM forms coming from unramified Hecke characters of $H$, but we have exhausted all the possibilities for those already. Possibly the unidentified parameters involve automorphic induction from $\GL_m(\A_H)$ with $m>1$. There are $2$-dimensional spaces of cusp forms of non-parallel odd weights $[7,5]$ and $[7,3]$, but they do not appear to be useful.
\end{remar}
\begin{remar}
 Looking in particular at the table for $N=8$, we can use the methods of previous sections to prove the guesses for global Arthur parameters in the cases $\mathbf{i=1,5,7,9,11,13,25,29}$. Note that, although the $\chi$-twists of these, namely $\mathbf{i=2,6,8,10,12,14,26,30}$ (for which thereby we also establish the global Arthur parameters), may appear to contradict Proposition \ref{eigentheta}(2), this is only if we assume that $\chi$-twisting preserves degrees. We can see in the simple case $N=6, \mathbf{i=1,2}$ that that assumption is false, since using $v_1={}^t(1,1,1,1,1,1)$ and $v_2={}^t(-1,1,-1,-1,1,1)$ and the automorphism group orders $82944,27648,$ $46080,$ $46080,$ $103680,103680$, one easily checks that $\Theta^{(1)}(v_1)$ is not a cusp form, while $\Theta^{(1)}(v_2)$ is, because $$-\frac{1}{82944}+\frac{1}{27648}-\frac{1}{46080}-\frac{1}{46080}+\frac{1}{103680}+\frac{1}{103680}=0.$$
 \end{remar}
 \begin{remar}
 We have not yet bothered to worry about whether $\SO_L/E_{\pp}$ is split and $\SO_L/\OO_{\pp}$ is reductive for every finite prime $\pp$. For $N=8$, both follow from the choice $L=E_8\otimes_{\Z}\OO_{\pp}$. For $N=4$ it is easy to prove at least that each $\SO_L/E_{\pp}$ is split, as in the proof of Lemma \ref{split}. But for $N=2$ or $6$ (or any odd multiple of $2$), choosing $L$ to be a direct sum of lattices with Gram matrix $\begin{pmatrix} 2 & \sqrt{3}\\\sqrt{3} & 2\end{pmatrix}$, the discriminant is $1$, whereas the discriminant is $-1$ for the direct sum of an odd number of hyperbolic planes (to which $L\otimes E_{\pp}$ would have to be equivalent for $\SO_L/E_{\pp}$ to be split), so in these cases $\SO_L/E_{\pp}$ is not split when $-1$ is not a square in $\FF_{\pp}$. As further confirmation that something is not quite right, we can observe in the tables that when $N=2$ and $N=6$ the conjectural Satake parameter $t_{\pp}(\pi_i)$, for $\chi(\pp)=-1$, is not in the image of $\SO_N(\C)$, having determinant $-1$. We may appear to have the same problem for some of the entries in the table for $N=8$, but closer inspection shows that this is not the case. For example, looking at $i=\mathbf{3}$, for $\pp$ such that $\chi(\pp)=-1$, the determinant of $t_{\pp}(\Delta_6)$ is also $-1$. To see this, note that since $\chi(\pp)=-1$, $\pp$ is inert in $\Q(\zeta_{12})/\Q(\sqrt{3})$. If $\pp\mid p$, a rational prime, then $p$ splits in $\Q(\sqrt{3})$, but not in $\Q(\sqrt{-3})$ (and $-1$ is not a square in $\FF_{\pp}$), because the compositum of these two fields is $\Q(\zeta_{12})$. Hence $\chi_{-3}(p)=-1$, but this is the same as $\det(t_{\pp}(\Delta_6))$, because $\Delta_6$ is the base change of the automorphic representation associated to a Hecke eigenform in $S_7(\Gamma_0(3),\chi_{-3})$. Alternatively, we could just use the fact that $\Delta_6$, coming from odd weight, must have central character $\chi$. Incidentally, this means it occurs not in $L^2(Z(\A_E)\GL_{2}(E)\backslash\GL_{2}(\A_E))$ (recall \S \ref{GAP}), rather in $L^2(\GL_{2}(E)\backslash\GL_{2}(\A_E), \chi)$.
\end{remar}
\begin{remar}
 There are various congruences of Hecke eigenvalues that can be explained by $23\mid(\zeta_E(4)/\pi^8)$ and $41\mid(\zeta_E(6)/\pi^{12})$.
\end{remar}

\section{Other fields}
According to \cite{Hs}, for $E=\Q(\sqrt{2})$ or $\Q(\sqrt{5})$ the rank of an even unimodular lattice must be divisible by $4$. In \cite[(1.2)]{Hs} a mass formula is applied to show that for $E=\Q(\sqrt{2})$ and rank $16$ the number of classes would be at least $2\times 10^{18}$, and for $E=\Q(\sqrt{5})$ and rank $16$ it would be at least $2\times 10^6$. For $E=\Q(\sqrt{2})$ and rank $12$ it would also appear to be very large. Thus for these fields the examples amenable to computation have now been dealt with. Similarly one can show that all plausible examples for $E = \Q(\sqrt{3})$ have been considered (the rank must be divisible by $2$ and the mass of the rank $10$ genus is large).

Naturally one asks how many other real quadratic fields and ranks are within reach. Necessarily we require a small number of classes in the corresponding genus and a mass formula can again be used to decide whether this is the case. If plausible we can then enumerate the classes by writing down one such lattice, taking iterated neighbours and testing for isometry (terminating once the mass is attained). As explained in the introduction, for $E=\Q(\sqrt{D})$ with $D\equiv 3\pmod{4}$, we may write down a rank $2$ free, even unimodular lattice of determinant $1$. Otherwise the recipes of \cite{Sc} can be used to write down a rank $4$ even unimodular lattice. Higher ranks can be reached by repeated orthogonal direct sums.
Strictly speaking, in the table below, for $N\equiv 2\pmod{4}$ not all the $0$ entries are for certain.

Using Magma we computed, for each real quadratic field $E$ of discriminant under $50$ and each rank $N\leq 12$, the number of classes in a genus of even unimodular lattices of determinant $1$ (where possible). See the table below. The symbol $T$ indicates that the mass of the genus is greater than $1$, so the number of classes is likely to be ``too large''. Note that the column for rank $4$ agrees with \cite[Theorem 3.4]{Sc}.

\begin{center}
\begin{tabular}{|c|c|c|c|c|c|c|}\hline Discriminant$\backslash$ Rank & $2$ & $4$ & $6$ & $8$ & $10$ & $12$\\\hline  $5^*$ & $0$ & $1$ & $0$ & $2$ & $0$ & $15$\\
$8^*$ & $0$ & $1$ & $0$ & $6$ & $0$ & $T$\\ $12$ & $1$ & $2$ & $6$ & $31$ & $T$ & $T$\\$13^*$ & $0$ & $1$ & $0$ & $12$ & $0$ & $T$\\$17^*$ & $0$ & $1$ & $0$ & $40$ & $0$ & $T$\\$21$ & $0$ & $3$ & $0$ & $T$ & $0$ & $T$\\$24$ & $0$ & $4$ & $0$ & $T$ & $0$ & $T$\\$28$ & $1$ & $4$ & $25$ & $T$ & $T$ & $T$\\$29^*$ & $0$ & $3$ & $0$ & $T$ & $0$ & $T$\\$33$ & $0$ & $4$ & $0$ & $T$ & $0$ & $T$\\$37^*$ & $0$ & $3$ & $0$ & $T$ & $0$ & $T$\\$40^*$ & $0$ & $6$ & $0$ & $T$ & $0$ & $T$\\
$41^*$ & $0$ & $3$ & $0$ & $T$ & $0$ & $T$\\$44$ & $1$ & $7$ & $T$ & $T$ & $T$ & $T$\\\hline
\end{tabular}
\end{center}

The fields are all of class number $1$, except for $\Q(\sqrt{10})$, for which the class number is $2$ (and for which we only considered the genus of a free even unimodular lattice of determinant $1$). Those marked with an asterisk have narrow class number equal to the class number.

The mass formula \cite[Lemma I.1]{Hs} shows that for a fixed rank, the number of classes grows at least polynomially with the discriminant of the field. Given this, it is extremely likely that the only examples of rank $6$ or higher that are computable are the ones in the table. For ranks $2$ and $4$ there are probably many more examples, but the variety of possible Arthur parameters is limited in these cases. For all cases in the table we have computed neighbour matrices and the data can be found at the second-named author's webpage \url{https://www.danfretwell.com/kneser}.

In the case $E=\Q(\sqrt{11})$ and $N=4$, where the number of classes is $7$, the following $7$ Arthur parameters appear to be correct. For various primes we tabulate the Hecke eigenvalues that these Arthur parameters would produce. These Hecke eigenvalues match the roots of the characteristic polynomials we computed. Although $5$ and $7$ split in $\OO_E$, the Hecke eigenalues are the same for both factors, so we give only the norms of the prime ideals. The bottom two rows of the table give the Hecke eigenvalues of the weight $2$ and $3$ forms we have used. As in the previous section, $\chi:E^{\times}\backslash\A_E^{\times}\rightarrow\C^{\times}$ is a character of conductor $\infty_1+\infty_2$.

\vskip5pt
\begin{center}
\begin{tabular}{|c|c|c|c|c|}\hline AP (conj'l)$\backslash \Nm(\pp)$ & $2$ & $5$ & $7$ & $9$\\\hline$[1]\oplus[3]$ & $9$ & $36$ & $64$ & $100$\\$\chi\otimes(")$ & $-9$ & $36$ & $-64$ & $100$\\$\Delta_1[2]$ & $3\sqrt{2}$ & $-6$ & $-16\sqrt{2}$ & $30$\\$\chi\otimes(")$ & $-3\sqrt{2}$ & $-6$ & $16\sqrt{2}$ & $30$\\$\Sym^2\Delta_1\oplus[1]$ & $2$ & $1$ & $8$ & $9$\\$\chi\otimes(")$ & $-2$ & $1$ & $-8$ & $9$\\$\Delta_2\oplus [1]\oplus\chi$ & $0$ & $9$ & $0$ & $25$\\\hline
$\Delta_1$ & $\sqrt{2}$ & $-1$ & $-2\sqrt{2}$ & $3$\\$\Delta_2$ & $0$ & $-1$ & $0$ & $7$\\\hline
\end{tabular}
\end{center}
\vskip5pt

The first $6$ of these automorphic representations of $\mathrm{O}(4)(\A_E)$ arise from tensor products, e.g. $\Sym^2\Delta_1$ is $\Delta_1\otimes\Delta_1$, as they ought to if they extend to $\mathrm{GO}(4)(\A_E)$, cf. \cite{Bo2}. But the last does not. This is possible because the $\mathrm{GO}(4)$-genus can contain a lattice whose discriminant is a totally positive but non-square unit, excluded from the $\mathrm{O}(4)$-genus.

In the case $E=\Q(\sqrt{7})$, $N=6$, where the number of classes is $25$, the following $15$ Arthur parameters likewise appear to be correct. The bottom $6$ rows of the table give the Hecke eigenvalues of the weight $2$, $3$, $4$, $4$, $5$ and $5$ forms we have used. For one of the weight $4$ forms, the coefficient field is of degree $4$, and we have listed just one of $4$ Galois conjugates. Note that $\chi\otimes\Delta_3^{(4)}$ is a Galois conjugate of $\Delta_3^{(4)}$.

\vskip5pt
\begin{center}
\begin{tabular}{|c|c|c|}\hline AP (conj'l)$\backslash \Nm(\pp)$ & $2$ & $3$\\\hline
$\chi\oplus[5]$ & $35$ & $112$\\$\chi\otimes(")$ & $35$ & $-112$\\$\Delta_3[2]\oplus[1]\oplus\chi$ & $23$ & $0$\\$\Delta_4\oplus[3]\oplus[1]$ & $19$ & $48$\\$\chi\otimes(")$ & $19$ & $-48$\\
$\Delta_4^a\oplus[3]\oplus[1]$ & $10$ & $48$\\$\chi\otimes(")$ & $10$ & $-48$\\$\Delta_3^{(4)}[2]\oplus[1]\oplus\chi$ & $2+6\sqrt{2}$ & $8\sqrt{10+3\sqrt{2}}$\\$\Delta_4\oplus\Delta_2\oplus[1]\oplus\chi$ & $3$ & $0$\\$\Delta_4^a\oplus\Delta_2\oplus[1]\oplus\chi$ & $-6$ & $0$\\$\Delta_4\oplus\Delta_1[2]$ & $-5$ & $0$\\$\Delta_4^a\oplus\Delta_1[2]$ & $-12$ & $0$\\\hline $\Delta_1$ & $-1$ & $0$\\$\Delta_2$ & $-3$ & $0$\\$\Delta_3$ & $5$ & $0$\\$\Delta_3^{(4)}$ & $-2+2\sqrt{2}$ & $2\sqrt{10+3\sqrt{2}}$\\$\Delta_4$ & $1$ & $0$\\$\Delta_4^a$ & $-8$ & $0$\\\hline
\end{tabular}
\end{center}
\vskip5pt

We did not look closely at the rank $8$ examples with $E=\Q(\sqrt{13})$ and $E=\Q(\sqrt{17})$. Data for these cases can be found at the webpage mentioned above.
 \section{Preliminaries on Hermitian lattices, even and unimodular over $\Z$}
 Let $E$ be an imaginary quadratic field, with ring of integers $\OO_E$, discriminant $-D$. For simplicity we shall suppose that the class number of $\OO_E$ is $1$. Let $L$ be an $\OO_E$-lattice in $V\simeq E^N$, with standard positive-definite $\OO_E$-integral Hermitian form $\x\mapsto \langle \x,\x\rangle={}^t\overline{\x} \x$. We may define a unitary group scheme $U_N$ over $\Z$, with $R$ points $U_N(R)=\{g\in M_N(R\otimes_{\Z}\OO_E)|\,\,{}^t\overline{g}g=I\}$. We assume that $L$ is even and unimodular as a rank-$2N$ $\Z$-lattice with the form $\tr_{E/\Q}(\langle,\rangle)$. We may define the (Hermitian) genus of $L$, algebraic modular forms (in) $M(\C,K)$, Hecke operators $T_{\pp}$, eigenvectors $v_i$ and automorphic representations $\pi_i$ of $U_N(\A)$, very much as before.

The theta series of degree $m$ of $L$ is
$$\theta^{(m)}(L,Z)=\sum_{\x\in L^m}\exp\left(\pi i\tr\left(\langle\x,\x\rangle Z\right)\right),$$ where $Z\in\mathcal{H}_m:=\{Z\in M_n(\C)|\,\,i({}^t\overline{Z}-Z)>0\}$. Then $\theta^{(m)}(L)\in M_{N}(U_{m,m}(\Z))$, by a theorem of Cohen and Resnikoff \cite{CR},\cite[Theorem 2.1]{HN}. Here $U_{m,m}(\Z):=\{g\in M_{2m}(\OO_E):\,\,{}^t\overline{g}Jg=J\}$, where $J=\begin{pmatrix} 0_m & -I_m\\I_m & 0_m\end{pmatrix}$. Thus, if $g=\begin{pmatrix} A & B\\C & D\end{pmatrix}\in U_{m,m}(\Z)$ then $$\theta^{(m)}(L,(AZ+B)(CZ+D)^{-1})=\det(CZ+D)^{N}\theta^{(m)}(L,Z).$$
 Again one can define linear maps $\Theta^{(m)}: M(\C,K)\rightarrow M_{N}(U_{m,m}(\Z))$ by
 $$\Theta^{(m)}\left(\sum_{j=1}^H y_je_j\right):=\sum_{j=1}^H\frac{y_j}{|\mathrm{Aut}(L_j)|}\,\theta^{(m)}(L_j).$$

 There is another way to construct the theta series $\theta^{(m)}(L,Z)$. Choosing a non-trivial additive character $\psi:\A/\Q\rightarrow\C^{\times}$ (and trivial multiplicative character $\chi:\A_E^\times/E^\times\rightarrow\C^\times$), consider the Weil representation $\omega=\omega_{\psi,\chi}$ of the group $U(m,m)(\A)\times U_N(\A)$ on the Schwartz space $\mathcal{S}(V(\A)^m)$. (See \cite[\S 1]{Ich2}.) Given $\Phi\in\mathcal{S}(V(\A)^m)$, we get a theta-kernel defined on $(g,h)\in U(m,m)(\A)\times U_N(\A)$,
 $$\Theta(g,h;\Phi):=\sum_{\x\in V^m(\Q)}\omega(g,h)\Phi(\x).$$
 If we choose $\Phi_{\infty}(\x):=\exp(\tr_{E/\Q}(\tr(\langle x_i,x_j\rangle)))$ and for finite $p$, $\Phi_p(\x):=\mathbbm{1}_{(L\otimes\Z_p)^m}(\x)$, then
 $$\Theta(g,h;\Phi)j(g,iI)^N=\theta^{(m)}(hL,Z),$$
 where $Z=g(iI)$, $hL$ is a lattice in the genus of $L$ and $j(g,iI)$ is a standard automorphy factor, which is just $(\det Y)^{-1/2}$ when for $Z=X+iY$ we take $g=\begin{pmatrix} C & X\,{}^tC^{-1}\\0 & {}^tC^{-1}\end{pmatrix}$, where $Y=C\,{}^tC$.

 If $dh$ is a measure on $U_N(\Q)\backslash U_N(\A)$ for which $U_N(\Q)\backslash U_N(\A)$ has volume $1$, then the theta integral
 $$I(g,\Phi):=\int_{U_N(\Q)\backslash U_N(\A)}\Theta(g,h;\Phi)\,dh=\frac{1}{\mu}\sum_{i=1}^H\frac{1}{\#\Aut(L_i)}\,\theta^{(m)}(L_i,Z),$$
 up to a factor $j(g,iI)^N$, where $\{L_i|\,1\leq i\leq H\}$ is a set of lattices representing the classes in the genus of $L$, and $\mu=\sum_{i=1}^H\frac{1}{\#\Aut(L_i)}$.
 More generally, if $y$ is a function on $U_N(\Q)\backslash U_N(\A)/K$, taking value $y_i$ on the class of $h_i$, where $h_iL=L_i$, then
 $$j(g,iI)^N\int_{U_N(\Q)\backslash U_N(\A)}\Theta(g,h;\Phi)\,y(h)\,dh=\frac{1}{\mu}\sum_{i=1}^H\frac{y_i}{\#\Aut(L_i)}\,\theta^{(m)}(L_i,Z),$$ which is the same as $\mu^{-1}\,\Theta^{(m)}\left(\sum_{j=1}^H y_je_j\right).$

 Consider the Eisenstein series $E^{(m)}(g,f_{\Phi}):=\sum_{\gamma\in P(\Q)\backslash U_{m,m}(\Q)}f_{\Phi}(\gamma g)$ (for $g\in U_{m,m}(\A)$), where $f_{\Phi}(g):=\omega(g,1)\Phi(\mathbf{0})$. (In the notation of \cite{Ich2} we have set $s=s_0$. In our notation, $s_0=\frac{N-m}{2}$.) It converges for $N>2m$, but can be defined for $N>m$ by a process of meromorphic continuation \cite[Lemma 8.2]{Ich1}. The following is part of a theorem of Ichino \cite[Theorem 1.1]{Ich2}.
 \begin{thm}\label{Ichino}
 If $N>m$ then $E^{(m)}(g,f_{\Phi})=I(g,\Phi)$.
 \end{thm}
 This is the Siegel-Weil formula, proved by Weil in the case $N>2m$ that the Eisenstein series converges.

 There will be local Langlands parameters $c_{\infty}(\pi_i):W_{\R}\rightarrow\GL_N(\C)\rtimes\Gal(E/\Q)$ and $c_{p}(\pi_i):W_{\Q_p}\rightarrow\GL_N(\C)\rtimes\Gal(E/\Q)$, which we always restrict to $W_{\C}$ and $W_{\pp}$, for each finite prime $\pp$ of $\OO_E$, and project to $\GL_N(\C)$ (with $\Frob_{\pp}\mapsto t_{\pp}(\pi_i)$). Necessarily
 $$c_{\infty}(\pi_i)(z)=\diag((z/\overline{z})^{(N-1)/2},\ldots,(z/\overline{z})^{-(N-1)/2})$$
 (up to conjugation in $\GL_N(\C)$). In the global Arthur parameters, instead of cuspidal automorphic representations of $\GL_{n_k}(\A)$, we see now cuspidal automorphic representations of $\GL_{n_k}(\A_E)$. For us, to say that $\pi_i$ has global Arthur parameter $A_i$ will now mean that $t_{\pp}(\pi_i)$ is conjugate in $\GL_N(\C)$ to $t_{\pp}(A_i)$ for all $\pp\nmid 2D$, and that $c_{\infty}(\pi_i)$ and $c_{\infty}(A_i)$, restricted to $\C^{\times}$, are conjugate in $\GL_N(\C)$. With the exclusion of $\pp\mid 2D$, this is a little weaker than what it might have meant.

 \begin{lem}\label{SL2toU11}
 If $E\neq\Q(i)$ then $U(1,1)(\Z)\simeq \OO_E^{\times}\times\SL_2(\Z)$.
 \end{lem}
 \begin{proof} Suppose that $g=\begin{pmatrix} a & b\\c & d\end{pmatrix}\in U(1,1)(\Z)$. Then $a\overline{d}-\overline{b}c=1$, $a\overline{c}=\overline{a}c$ and $b\overline{d}=\overline{b}d$. The first equation implies that $a,c$ (and likewise $\overline{a},\overline{c}$) are coprime, then the second implies that $a,\overline{a}$ are associates, say $a=u\overline{a}$, with $u\in\OO_E^{\times}$. The second equation implies also that $c=u\overline{c}$. Similarly using the third equation (and conjugating the first to see that it must be the same unit involved), we find that also $b=u\overline{b}$ and $d=u\overline{d}$. Since $E\neq\Q(i)$, either $u$ or $-u$ is a square. In the latter case, say $u=-v^2$, then $a/v=-\overline{a/v}$, which implies that $a/v$ is an integer multiple of $\sqrt{-D}$ (or $\sqrt{-D/2}$), where $-D$ is the discriminant of $E/\Q$. Likewise for all the other entries, but then the determinant of $g$ fails to be a unit, so we must be in the case $u=v^2$, so $a/v=\overline{a/v}$ is in $\Z$, and likewise for all the other entries.
 \end{proof}

 \begin{prop}\label{eigenthetaHerm}
 \begin{enumerate}
 \item If $v_i\in M(\C,K)$ is an eigenvector for $H_K$, then $\Theta^{(m)}(v_i)$ (if non-zero) is a Hecke eigenform, at least away from $p\mid 2D$.
 \item Suppose that $\Theta^{(m)}(v_i)$ is non-zero, and that $N\geq 2m$. Let $$t_\pp(\pi_i)=\diag(\beta_{1,\pp},\ldots,\beta_{N,\pp})$$ be the Satake parameter at $\pp$ for $v_i$ (with $\pp\nmid 2D$), and $\diag(\alpha_{1,\pp},\ldots,\alpha_{2m,\pp})\in\GL_{2m}(\C)$ the Satake parameter at $\pp$ of the automorphic representation of $U_{m,m}(\A)$ generated by $\Theta^{(m)}(v_i)$. Then, as multisets,
     $$\{\beta_{1,\pp},\ldots,\beta_{N,\pp}\}=$$
     $$\begin{cases} \{\alpha_{1,\pp},\ldots,\alpha_{2m,\pp}\}\cup \{\Nm\pp^{(N-2m-1)/2},\ldots,\Nm\pp^{-(N-2m-1)/2}\} & \text{if $N>2m$};\\\{\alpha_{1,p},\ldots,\alpha_{2m,p}\} & \text{ if $N=2m$. }\end{cases}$$
     \item If $4\mid N$ and $N>2m$ then a cuspidal Hecke eigenform $F\in S_N(U_{m,m}(\Z))$ is in the image of $\Theta^{(m)}$ if $L(\st,F,(N+1-2m)/2)\neq 0$, where $L(\st,F,s)=\prod_\pp\prod_{i=1}^{2m}(1-\alpha_{i,\pp}\Nm\pp^{-s})^{-1}$ is the standard $L$-function.
 \end{enumerate}
 \end{prop}
 \begin{proof}
 Since $L$ is self-dual as an Hermitian lattice locally at all $p\nmid D$, and since Hermitian $\OO_E\otimes_\Z\Z_p$-lattices are determined up to isometry by their invariant factors
 \cite[Prop. 3.2]{J},\cite[Thm. 7.1]{Sh2}, the stabiliser in $U_N(\Q_p)$ of $\OO_E\otimes_\Z\Z_p$ is isomorphic to the standard $U_{N/2,N/2}(\Z_p)$. It follows that (1) and (2) are covered by work of Y. Liu \cite[Appendix]{Liu}, which also requires $p\neq 2$. See also \cite[Prop. 2.1]{Ich1}.

 Now we turn to (3). The doubling method of Piatetski-Shapiro and Rallis was developed in the case of unitary groups by Li \cite{Li} and by Harris, Kudla and Sweet \cite{HKS}. Consider the Eisenstein series $E^{(2m)}(g,\Phi)$ as above. We may embed $U_{m,m}\times U_{m,m}$ into $U_{2m,2m}$ as in \cite{HLS}, thus write $E^{(2m)}(g_1,g_2,\Phi)$ for $(g_1,g_2)\in U_{m,m}(\A)\times U_{m,m}(\A)$. To $F$ we may associate a function $\phi_F$ on $U_{m,m}(\A)$ in a standard way. By \cite[(3.1.2.8)]{HLS} (``Basic identity of Piatetski-Shapiro and Rallis''), with $s=s_0=\frac{N-2m}{2}$ and $\chi$ trivial (noting that $2m$ has been substituted for $m$ compared to above), we find that
     $$\int_{U_{m,m}(\Q)\backslash U_{m,m}(\A)}\int_{U_{m,m}(\Q)\backslash U_{m,m}(\A)}E^{(2m)}(g_1,g_2,\Phi)\phi_F(g_1)\phi_{\overline{F}}(g_2)\,dg_1\,dg_2$$
     is equal to
     $$L_D(\st,F,(N+1-2m)/2)Z_{\infty}(s_0,F,\Phi)\prod_{p\mid D}Z_p(s_0,F,\Phi)\left(\prod_{r=0}^{m-1}L(N-m-r,\chi_{-D}^{r})\right)^{-1},$$
     where $-D$ is the discriminant of $\OO_E$, the subscript $D$ means we omit Euler factors at primes $p\mid D$, $Z_{\infty}(s_0,F,\Phi)$ and the $Z_p(s_0,F,\Phi)$ are certain local zeta integrals and $\chi_{-D}$ is the quadratic character associated to $E/\Q$. We have corrected the power of $\chi_{-D}$, as in the footnote on \cite[p.42]{EHLS}. We need to know that $Z_{\infty}(s_0,F,\Phi)\prod_{p\mid D}Z_p(s_0,F,\Phi)\neq 0$. For $p\mid D$ an argument of Lanphier and Urtis \cite[\S 4, case $v\nmid\mathfrak{n}$]{LU} shows that
     $Z_p(s_0,F,\Phi)\neq 0$. To justify this, note that even though $U_{m,m}$ is ramified at such $p$, the maximal compact subgroup $U_{m,m}(\Z_p)$ is special (if not hyperspecial), as noted in \cite[\S 2.1]{AK}, so spherical vectors are still unique up to scaling \cite[\S 2.3]{Min}. We may also call on \cite[\S 4]{LU} for the non-vanishing of $Z_{\infty}(s_0,F,\Phi)$. We may now proceed as in the proof of \cite[Theorem 3]{LU}. This is close to B\"ocherer's idea of using the Siegel-Weil formula to substitute for the Eisenstein series in a pull-back formula/doubling integral \cite{Bo}, and our condition $N>2m$ is in order to apply Theorem \ref{Ichino}, with $2m$ substituted for $m$ because of the doubling.
  \end{proof}
 \begin{prop}\label{IkMiyHerm}
 Suppose that $\OO_E$ has class number $1$, and let $w$ be the number of units in $\OO_E$. Let $\kappa, g$ be even natural numbers, and suppose that $w\mid (\kappa/2)$.
 \begin{enumerate}
 \item Let $f\in S_{\kappa-g+1}(\Gamma_0(D),\chi_{-D})$ be a Hecke eigenform, where $E=\Q(\sqrt{-D})$ has discriminant $-D$ and $\chi_{-D}$ is the associated quadratic character. Assume that $f$ is not a CM form coming from a Hecke character of $K$. Then there exists a Hecke eigenform $F\in S_{\kappa}(U_{g,g}(\Z))$ with standard $L$-function $$L(\st,F,s)=\prod_{i=1}^g L(f,s+\frac{\kappa+1}{2}-i)L(f,s+\frac{\kappa+1}{2}-i,\chi_{-D}).$$
 \item Let $G\in S_{\kappa}(U_{r,r}(\Z))$ be a Hecke eigenform, for $r<g$. For $F$ as above, the function $$\FFF_{f,G}(Z):=\int_{U_{r,r}(\Z)\backslash\mathcal{H}_r}F\left(\begin{pmatrix} Z & 0\\0 & W\end{pmatrix}\right)\overline{G(W)}(\det\Im W)^{\kappa-2r}\,dW,$$ if non-zero, is a Hecke eigenform in $S_{\kappa}(U_{g-r,g-r})$, with standard $L$-function (if $\kappa\geq 2(g-r)$) $$L(\st,\FFF_{f,G},s)=L(\st,G,s)\prod_{i=1}^{g-2r}L(f,s+\frac{\kappa-2r+1}{2}-i)L(f,s+\frac{\kappa-2r+1}{2}-i,\chi_{-D}).$$
 \end{enumerate}
 \end{prop}
 (1) was proved by Ikeda, and follows from Theorem 5.2, Corollary 15.21 and Theorem 18.1 in \cite{I1}. (2) is a theorem of Atobe and Kojima \cite[Theorem 1.1]{AK}. For simplicity we have imposed unnecessary conditions that are satisfied in our application.

\section{$12$-dimensional Hermitian forms over $\Q(\sqrt{-3})$, even unimodular over $\Z$}
When $E=\Q(\sqrt{-3})$ and $L$ is an $\OO_E$-lattice in $E^N$, even and unimodular as a $\Z$-lattice, $8\mid 2N\implies 4\mid N$. There is a single genus of such lattices \cite[Remark 1]{HKN}. When $N=4$ or $8$ there is a single class in the genus \cite[Corollary 1]{HKN}, and the global Arthur parameter will be $[N]$. For $N=12$, the genus contains $5$ classes,
studied by Hentschel, Krieg and Nebe \cite{HKN}. The matrix representing $T_{(2)}$, with respect to the basis ordered as in \cite[Theorem 2]{HKN}, is $\begin{pmatrix} 65520 & 3888000 & 1640250 & 0 & 0\\ 1458 & 516285 & 3956283 & 1119744 & 0\\15 & 96480 & 2467899 & 2998272 & 31104\\0 & 13365 & 1467477 & 3935781 & 177147\\0 & 0 & 405405 & 4717440 & 470925\end{pmatrix}$. This was computed by S. Sch\"onnenbeck, and given in \cite[\S 3.3]{DS}

\vskip5pt
\begin{center}
\begin{tabular}{|c|c|c|c|}\hline$\mathbf{i}$ & $\lambda_i\left(T_{(2)}\right)$ &  $g_i$ & Global Arthur parameters\\\hline
 $\mathbf{1}$ & $5593770$ & $0$ & $[12]$\\$\mathbf{2}$ & $1395945$ & $1$ & $\Delta_{11}\oplus [10]$\\$\mathbf{4}$ & $357525$ & $2$ & ${}^3\Delta_{10}[2]\oplus [8]$\\$\mathbf{8}$ & $85365$ & $3$ & $\Delta_{11}\oplus{}^3\Delta_8[2]\oplus [6]$\\$\mathbf{9}$ & $23805$ & $4$ & ${}^3\Delta_8[4]\oplus [4]$\\\hline
 \end{tabular}
 \end{center}
\vskip5pt
Some of the notation is further explained during the proof of the proposition below. In \cite{DS} we looked at a genus of $20$ classes of rank $12$ Hermitian $\OO_E$-lattices, unimodular as Hermitian (rather than Euclidean) lattices, and conjectured global Arthur parameters for all the eigenvectors arising. The entries in the above table match $5$ of those in \cite{DS}, and we have preserved the numbering used there, hence the gaps.

The eigenvectors are $$v_1=\begin{pmatrix}1\\1\\1\\1\\1\end{pmatrix}, v_2=\begin{pmatrix}-6000\\-1854\\-472\\219\\910\end{pmatrix}, v_4=\begin{pmatrix} 648000\\49572\\-2144\\-297\\20020\end{pmatrix}, v_8=\begin{pmatrix}-8294400\\-46656\\10240\\-7425\\80080\end{pmatrix}
v_9=\begin{pmatrix}6220800\\-69984\\7680\\-4455\\40040\end{pmatrix}.$$ Using the sizes of automorphism groups from \cite[Theorem 2]{HKN}, we find then that $\Theta^{(m)}(v_9)$ is a scalar multiple of $$\theta^{(m)}(L_1)-30\,\theta^{(m)}(L_2)+135\,\theta^{(m)}(L_3)-160\,\theta^{(m)}(L_4)+54\,\theta^{(m)}(L_5),$$ in agreement with the linear combination in \cite[Theorem 3(a)]{HKN}.
Note that the other linear combinations there do not correspond to eigenvectors, since they only represent quotients in a filtration.
\begin{prop}\label{propherm} The global Arthur parameters and degrees are as in the table.
\end{prop}
\begin{proof}

$\mathbf{i=1}$. Similar to earlier examples, we can get this from ``Siegel's Main Theorem'' (a.k.a. Siegel-Weil formula), as stated in \cite[Corollary 3]{HKN}.

$\mathbf{i=2}$. Let $\Delta=\sum_{n=1}^{\infty}\tau(n)q^n=q-24q^2+252q^3\ldots$ be the normalised cusp form spanning $S_{12}(\SL_2(\Z))$. Using Lemma \ref{SL2toU11} and $\#\OO_E^{\times}\mid 12$, the function $\Delta$ on $\HH_1=\mathcal{H}_1$ belongs to $S_{12}(U(1,1)(\Z))$. Since $12>2$ and $L(\st,\Delta,11/2)=L(\Delta,11)L(\Delta,\chi_{-3},11)\neq 0$, Proposition \ref{eigenthetaHerm}(3) implies that $\Delta=\Theta^{(1)}(v_i)$ for some $i$. By Proposition \ref{eigenthetaHerm}(2), $\pi_i$ has global Arthur parameter $\Delta_{11}\oplus [10]$ (where $\Delta_{11}$ is now the base change to $\GL_2(\A_E)$ of that appearing in \S 3). Hence $\lambda_i(T_{(2)})=((-24)^2-2\cdot2^{11})+4\,\frac{4^{10}-1}{4-1}+\frac{2^{12}-1}{2+1}=1395945$, as in \cite[Proposition 4.1]{DS}, so $i=2$.

$\mathbf{i=4}$. The space $S_{11}(\Gamma_0(3),\chi_{-3})$ is $2$-dimensional, spanned by a Hecke eigenform $g=q+12\sqrt{-5} q^2$
 $$+(-27+108\sqrt{-5})q^3+304q^4-1272\sqrt{-5}q^5+(-6480-324\sqrt{-5})q^6+17324 q^7+\ldots$$
 and its (Galois or complex) conjugate. The associated cuspidal automorphic representations of $\GL_2(\A)$ are quadratic twists by $\chi_{-3}$ of one another, so share the same base change to $\GL_2(\A_E)$, which we denote ${}^3\Delta_{10}$. We apply Proposition \ref{IkMiyHerm}(1), with $\kappa=12, g=2$ (so $\kappa-g+1=11$) to produce an Hermitian Ikeda lift $F=I^{(2)}(g)$. Since
 $L(\st,F,9/2)=L(g,10)L(g,9)L(g,10,\chi_{-3})L(g,9,\chi_{-3})\neq 0$, Proposition \ref{eigenthetaHerm}(3) shows that $F=\Theta^{(2)}(v_i)$ for some $i$, and it follows from Proposition \ref{eigenthetaHerm}(2) and $L(\st,F,s)=L(g,s+11/2)L(g,s+11/2,\chi_{-3})L(g,s+9/2)L(g,s+9/2,\chi_{-3})$ that $\pi_i$ has global Arthur parameter ${}^3\Delta_{10}[2]\oplus [8]$.
 Then $\lambda_i(T_{(2)})=((12\sqrt{-5})^2+2\cdot2^{10})(1+4)+4^2\,\frac{4^{8}-1}{4-1}+\frac{2^{12}-1}{2+1}=357525$, so $i=4$.

$\mathbf{i=9}$. The space $S_{9}(\Gamma_0(3),\chi_{-3})$ is $2$-dimensional, spanned by a Hecke eigenform $f=q+6\sqrt{-14} q^2+\ldots$
 and its (Galois or complex) conjugate. We proceed as for $i=4$, now with $\kappa=12, g=4$, so $\kappa-g+1=9$, showing in the process that $g_9=4$ and $\Theta^{(4)}(v_9)=I^{(4)}(f)$, which was conjectured in \cite[Remark 3(b)]{HKN}. (The other degrees were proved in \cite[Theorem 3]{HKN} by computing coefficients of theta series.)

$\mathbf{i=8}$. Since $(v_9,v_8\circ v_2)\neq 0$ and we already know that $g_9=4$ and $g_2=1$, Corollary \ref{degbound} implies that $g_8\geq 3$. Then $(v_8,v_4\circ v_2)\neq 0$ gives $g_8\leq 3$, so $g_8=3$. Now $(v_9,v_8\circ v_2)\neq 0$ tells us, via Proposition \ref{nonzero}, that $\langle\FFF_{f,\Delta},\Theta^{(3)}(v_8)\rangle\neq 0$, so that $\Theta^{(3)}(v_8)$ and $\FFF_{f,\Delta}$ lie in the same Hecke eigenspace, and have the same standard $L$-function. Then using $L(\st,\FFF_{f,\Delta},s)=$
$$L(\Delta,s+\frac{11}{2})L(\Delta,s+\frac{11}{2},\chi_{-3})L(f,s+\frac{9}{2})L(f,s+\frac{9}{2},\chi_{-3})L(f,s+\frac{7}{2})L(f,s+\frac{7}{2},\chi_{-3})$$
and Proposition \ref{eigenthetaHerm}(2), we deduce that $\pi_8$ has global Arthur parameter $\Delta_{11}\oplus{}^3\Delta_8[2]\oplus [6]$. Incidentally, of course the value of $\lambda_8(T_{(2)})$ implied by this agrees with that computed using neighbours. We can go further, now we know that $\FFF_{f,\Delta}\neq 0$. Since $N=12$ and $m=3$, $N>2m$, so Proposition \ref{eigenthetaHerm} applies, and
$$L(\st,\FFF_{f,\Delta},(N+1-2m)/2)$$
$$=L(\Delta,9)L(\Delta,9,\chi_{-3})L(f,8)L(f,8,\chi_{-3})L(f,7)L(f,7,\chi_{-3})\neq 0,$$
so $\FFF_{f,\Delta}$ is in the image of $\Theta^{(3)}$, necessarily a scalar multiple of $\Theta^{(3)}(v_8)$.
\end{proof}
\begin{remar} It follows from the above that, up to scalar multiples,
\begin{enumerate}
\item either Hecke eigenform $g\in S_{11}(\Gamma_0(3),\chi_{-3})$ has degree $2$ Hermitian Ikeda lift
$$I^{(2)}(g)=\theta^{(2)}(L_1)+840\,\theta^{(2)}(L_2)-1206\,\theta^{(2)}(L_3)-1024\,\theta^{(2)}(L_4)+2592\,\theta^{(2)}(L_5);$$
\item either Hecke eigenform $f\in S_9(\Gamma_0(3),\chi_{-3})$ and $\Delta\in S_{12}(\SL_2(\Z))$ have degree $3$ Hermitian Miyawaki lift
$$\FFF_{f,\Delta}=\theta^{(3)}(L_1)+15\,\theta^{(3)}(L_2)-135\,\theta^{(3)}(L_3)+200\,\theta^{(3)}(L_4)-81\,\theta^{(3)}(L_5);$$
\item $f$ as above has degree $4$ Hermitian Ikeda lift
$$I^{(4)}(f)=\theta^{(4)}(L_1)-30\,\theta^{(4)}(L_2)+135\,\theta^{(4)}(L_3)-160\,\theta^{(4)}(L_4)+54\,\theta^{(4)}(L_5).$$
\end{enumerate}
For any fixed $m$ with $0\leq m\leq 4$, the $\theta^{(m)}(v_i)$ such that $g_i\leq m$ are linearly independent. For the unique $i$ with $g_i=m$, $\theta^{(m)}(v_i)$ is a cusp form, while those $\theta^{(m)}(v_i)$ with $g_i<m$ are killed by different powers of the Siegel operator.
\end{remar}

\end{document}